\documentclass[12pt]{amsart}

\usepackage[usenames,dvipsnames,svgnames,table]{xcolor}

\usepackage[margin=1.2in]{geometry}
\usepackage{amssymb}
\usepackage{amsthm}
\usepackage[shortlabels]{enumitem}
\usepackage{tikz-cd}
\usetikzlibrary{arrows.meta, shapes.geometric, positioning}

\usepackage[normalem]{ulem}
\usepackage{booktabs}

\usepackage{hyperref}
\hypersetup{
    colorlinks=true,
    linkcolor=blue,
    filecolor=magenta,      
    urlcolor=cyan
    }
\usepackage[capitalize, noabbrev]{cleveref}

\numberwithin{equation}{section}

\newtheorem{theorem}[equation]{Theorem}
\newtheorem{question}[equation]{Question}
\newtheorem{conjecture}[equation]{Conjecture}

\newtheorem{lemma}[equation]{Lemma}

\newtheorem{corollary}[equation]{Corollary}
\newtheorem{proposition}[equation]{Proposition}
\newtheorem{problem}[equation]{Problem}

\theoremstyle{remark}
\newtheorem{example}[equation]{Example}
\newtheorem{remark}[equation]{Remark}

\theoremstyle{definition}
\newtheorem{definition}[equation]{Definition}

\DeclareMathOperator{\projection}{pr}

\DeclareMathOperator{\basic}{basic}

\DeclareMathOperator{\ham}{ham}

\DeclareMathOperator{\DuHe}{DH}
\DeclareMathOperator{\dR}{dR}

\def \calO {{\mathcal O}}
\def \t {{\mathfrak{t}}}
\def \h {{\mathfrak{h}}}
\def \k {{\mathfrak{k}}}
\def \R {{\mathbb R}}

\def \Z {{\mathbb Z}}
\def \Q {{\mathbb Q}}

\def \on {\operatorname}

\def \actson {\  \rotatebox[origin=c]{-90}{$\circlearrowright$}\  }

\newcommand{\bigslant}[2]{{\raisebox{.2em}{$#1$}\left/\raisebox{-.2em}{$#2$}\right.}}

\usepackage[
    style=numeric,
    doi=false,
    isbn=false,
    url=false,
    eprint=false,
    sorting=nyt,
    maxbibnames=8,
	backend=biber
]{biblatex}

\title[]{Symplectic torus actions with \\ non-contractible orbits}
\author{Rei Henigman}
\address{School of Mathematical Sciences, Tel Aviv University}
\email{rei.henigman@gmail.com}
\author{Yael Karshon}
\address{School of Mathematical Sciences, Tel Aviv University; and the University of Toronto}
\email{yaelkarshon@tauex.tau.ac.il}



\date{\today}

\begin{document}

\begin{abstract}
We prove that a symplectic $T^{n-1}$ action on a closed connected
$2n$-dimensional symplectic manifold is Hamiltonian if and only if its
orbits are contractible. This generalizes a result of Lalonde--McDuff--Polterovich on
four-manifolds and theorems of McDuff and Kim on existence of fixed
points. 
When the orbits are isotropic, we prove a stronger variant of this result,
which implies
non-extendability of certain symplectic circle actions on 6-manifolds
to symplectic $T^2$ actions.
Moreover, we prove that a symplectic $T^{n-1}$ action with isotropic orbits 
always splits into a maximal Hamiltonian action and a locally-free action. 
We end by posing several open questions on the topology of symplectic torus actions.
\end{abstract}
\maketitle

\setcounter{tocdepth}{1}
{
  \hypersetup{linkcolor=black}
  \tableofcontents
}

\section{Introduction}
\label{sec:intro}

\subsection{Hamiltonian versus non-Hamiltonian} \ 
\label{subsec:ham-vs-nonham}

Let a torus $T$ act symplectically and faithfully on a closed connected symplectic manifold $(M^{2n}, \omega)$. When the action is Hamiltonian, many properties of the $T$-manifold can be read from its momentum map, its fixed point set $M^T$, and its isotropy data. In particular, one can read the fundamental group, the Betti numbers, the volumes of the reduced spaces, and information on the Chern class. Hence, it is natural to seek conditions under which a symplectic torus action must be Hamiltonian.

For a symplectic circle action, the existence of fixed points implies that the action is Hamiltonian under various assumptions:
Kähler manifolds (Frankel~\cite{frankel}), weak Lefschetz property 
(Ono~\cite{ono}), monotone symplectic manifolds 
(Lupton--Oprea~\cite{lupton_oprea}), semi-free actions with finite fixed set $M^T$
(Tolman--Weitsman~\cite{weitsman_tolman}), and 
symplectic four-manifolds (McDuff~\cite[Proposition~2]{mcduff_circle_actions}).
However, this is not true in general: McDuff~\cite{mcduff_circle_actions} constructed a symplectic 
non-Hamiltonian circle action on a six-dimensional manifold whose 
fixed-point set is a nonempty union of 2-tori.
In~\cite[p.~151]{introduction_symplectic_first_edition}, 
McDuff and Salamon asked the following question, often referred to as 
\textit{McDuff's Conjecture}: 
``Does there exist a non-Hamiltonian symplectic circle action with finite
fixed set $M^T$ on a closed, connected symplectic manifold?''. This question 
was resolved by Tolman~\cite{tolman_isolated_points}, who constructed a 
non-Hamiltonian symplectic circle action on a six-dimensional manifold, 
with exactly $32$ fixed points (see also Jang--Tolman~\cite{jang_tolman}).

In this paper, we are interested in a related question:
\emph{contractibility} of the $T$-orbits.
This is closely related to the orbit maps $\rho_x \colon T \to M$,
given by $a \mapsto a \cdot x$ for $x \in M$, 
being null-homotopic. Because $M$ is connected, any two orbit maps are homotopic,
so if there exist fixed points, 
then all the orbit maps are null-homotopic.
Hence, for a Hamiltonian torus action 
on a closed connected symplectic manifold, the orbit maps are null-homotopic.
For Hamiltonian circle actions, a vast generalization, proved using Floer theory,
is that the evaluation at a point $p \in M$ of a loop of Hamiltonian diffeomorphisms 
is null-homotopic (Lalonde--McDuff--Polterovich \cite[Section 1.2]{lalonde_mcduff_polterovich}).

For non-Hamiltonian symplectic torus actions, the situation
is almost the opposite, as we now describe.
The action of $\R^2/\Z^2$ on itself by translations is free,
and its orbit maps are not null-homotopic.
On the Kodaira-Thurston symplectic 4-manifold,
there exists a free circle action whose orbits are null-homologous,
but its orbit maps are not null-homotopic.
In fact, in dimension four, the orbit maps of symplectic non-Hamiltonian
circle actions can never be null-homotopic, by a work of Lalonde-McDuff-Polterovich~\cite{lalonde_mcduff_polterovich}. This extends to dimensions higher than four under various additional assumptions:
weak Lefschetz property (Allday--Oprea~\cite[Proposition 3.4]{allday_oprea}), 
aspherical symplectic manifolds (Lupton--Oprea~\cite{lupton_oprea}),
symplectically aspherical symplectic manifolds 
(Oprea--Walsh~\cite[Cor.~4.10]{oprea_walsh}), 
and spherically monotone symplectic manifolds
(Atallah~\cite{atallah_remarks_on_circle_actions}).
In contrast, the 6-dimensional examples of McDuff~\cite{mcduff_circle_actions}, Tolman~\cite{tolman_isolated_points}, and Jang-Tolman~\cite{jang_tolman} are non-Hamiltonian circle actions with fixed points, hence with null-homotopic orbit maps. See also Kotschick~\cite{kotschick_contractible_orbits} for examples of free symplectic (non-Hamiltonian) circle actions with contractible orbits in every even dimension $\geq 6$ (which resolved a question of McDuff--Salamon \cite[p.~152]{introduction_symplectic_first_edition}).

Taking products of the above 6-dimensional examples with spheres, we obtain symplectic non-Hamiltonian $T^{n-2}$ actions on $(M^{2n}, \omega)$ with null-homotopic orbit maps for every $n$. In contrast, it follows from works of Duistermaat-Pelayo~\cite{duistermaat_pelayo} and Benoist~\cite{benoist} that a symplectic non-Hamiltonian $T^n$ action on $(M^{2n}, \omega)$ cannot have null-homotopic orbit maps (see Appendix~\ref{weak_main_theorem_appendix}). Hence, the only missing case is symplectic $T^{n-1}$ actions, which we resolve in this paper, thereby extending Lalonde-McDuff-Polterovich's result~\cite{lalonde_mcduff_polterovich} from dimension four to any dimension:

\begin{theorem} \label{thm:main}
Let a torus $T$ act symplectically on a closed connected 
symplectic manifold $(M, \omega)$.  
Suppose that $\dim M = 2n$ and $\dim T \geq n-1$.
Then the action is Hamiltonian if and only if its orbit maps are null-homotopic.
\end{theorem}

We prove \cref{thm:main} in \cref{proof_of_main_thm_section}.
The main case, when $\dim T = n-1$ and the action is non-Hamiltonian,
is in \cref{complexity_one_proposition}.

When $\dim M = 4$ and $\dim T = 1$, this is precisely
Lalonde-McDuff-Polterovich's theorem~\cite{lalonde_mcduff_polterovich}. Their argument in dimension four
strongly relies on the two-dimensionality of the reduced spaces with respect to a
circle-valued momentum map associated to the locally-free non-Hamiltonian circle action. For $\dim T > 1$, the action might have
both Hamiltonian and non-Hamiltonian directions, so their argument no longer works.
Our argument uses the log-concavity of the Duistermaat-Heckman measure for complexity one spaces to deduce that the ``non-Hamiltonian part'' of the $T$-action acts in a locally-free manner (for $\dim T \ge n-1$). Then, the reduced spaces with respect to a cylinder-valued momentum map associated to the locally-free ``non-Hamiltonian part'' of the $T$-action are complexity-one orbifolds, and we obtain our result by applying Karshon-Tolman's Moser-type argument for complexity one spaces.

An \textbf{orbit} of the $T$-action is \textbf{contractible} if its \emph{inclusion map} is null-homotopic. 
In general, this is stronger than the orbit map being null-homotopic;
see Example~\ref{example_fixed_non_contractible}.
But in the setup of \cref{thm:main}, these conditions become equivalent:
\begin{corollary}\label{corollary_tfae} 
Let a torus $T$ act symplectically on a nonempty
closed connected symplectic manifold $(M, \omega)$.
Suppose that $\dim M = 2n$ and $\dim T \geq n-1$.
Then the following are equivalent.
    \begin{enumerate}[topsep=0pt]
        \item\label{ham_item} The action is Hamiltonian.
        \item\label{fixed_item} There exist fixed points.
        \item The orbit maps are null-homotopic.
        \item Some orbits are contractible.
        \item All orbits are contractible.
        \item The action is equivariantly formal.
    \end{enumerate}
\end{corollary}
\cref{corollary_tfae} follows from the following diagram of implications, which summarize the relations between the different properties:
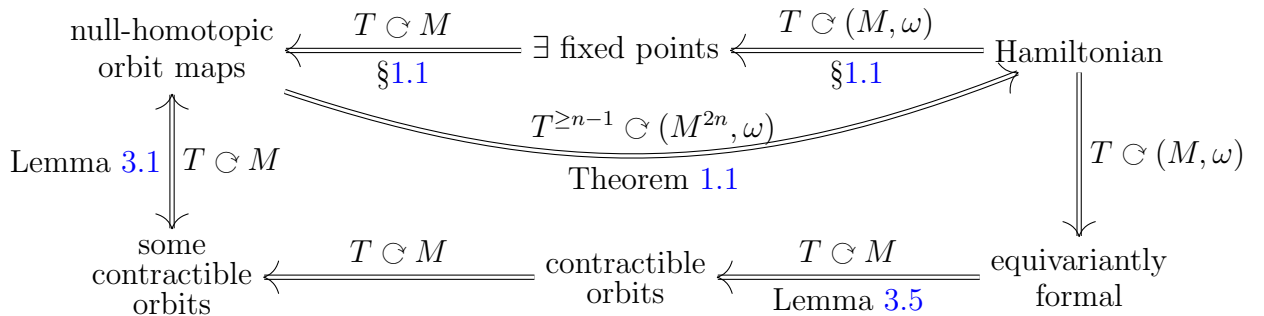
\begin{center}
    \begin{tikzpicture}[
        double distance=1.2pt,
        line width=0.4pt
    ]
    \node (null) at (-6,0) {\shortstack{null-homotopic\\orbit maps}};
    \node (fixed) at (0,0) {$\exists$ fixed points};
    \node (ham) at (6, 0) {Hamiltonian};
    \node (1contr) at (-6,-3) {\shortstack{some \\ contractible\\orbits}};
    \node (contr) at (0,-3) {\shortstack{contractible\\orbits}};
    \node (formal) at (6,-3) {\shortstack{equivariantly\\formal}};

    \draw[->, double] (fixed) -- node[below]{\S \ref{subsec:ham-vs-nonham}} 
    node[above]{$T \actson M$} (null);
    \draw[->, double] (ham) -- node[below]{\S \ref{subsec:ham-vs-nonham}} 
    node[above]{$T \actson (M,\omega)$} (fixed);
    \draw[->, double] (ham) -- 
    node[right]{$T \actson (M,\omega)$} (formal);
    \draw[->, double] (formal) -- node[below]{\cref{hamiltonian_has_contractible}} node[above]{$T \actson M$} (contr);
    \draw[->, double] (contr) -- node[above]{$T \actson M$} 
    (1contr);
    \draw[<->, double] (1contr) -- node[left]{\cref{lem:contractible}} 
    node[right]{$T \actson M$} (null);
    \draw[->, bend right=20, double] (null) to node[below] {\cref{thm:main}}
    node[above]{$T^{\geq n-1} \actson (M^{2n},\omega)$} (ham);

    \end{tikzpicture}
\end{center}

In particular, \cref{hamiltonian_has_contractible} implies that every orbit of a Hamiltonian torus action is contractible, regardless of the dimension. This resolves a question of Oprea and Walsh~\cite{oprea_walsh}, see \cref{contractible_section} for more information. For symplectic non-Hamiltonian torus actions with  $\dim T < n-1$, it is unknown whether ``all orbits are contractible'' is equivalent to ``null-homotopic orbit maps'', see \cref{fixed_non_contractible_question} and \cref{some_non_contractible_question}.

In the setup of \cref{corollary_tfae},
``$\exists$ fixed points $\implies$ Hamiltonian''
was also proved by Kim~\cite{kim_frankel_complexity_one},
and can also be deduced from \cref{splitting_lemma}
(``Splitting lemma'') and \cref{fixed implies isotropic}.
``Hamiltonian $\implies$ equivariantly formal'' 
is true regardless of the dimension of $T$.  
This is due to Atiyah--Bott and Kirwan
(\cite[Sect. 7]{atiyah_bott}, \cite[Prop. 5.8]{kirwan}).
The converse is also true by Allday--Hauschild--Puppe~\cite{allday_hauschild_puppe};
see also a direct proof in Appendix~\ref{equivariant_formality_appendix}.

\subsection{Actions with isotropic orbits and non-extendability} \ 
\label{subsec:isotropic_extensions_section}

An orbit of the $T$-action is \textbf{isotropic} if $\omega$ vanishes on the orbit.
A central step in the proof of \cref{thm:main} 
is the following stronger result for actions with isotropic orbits.
(If one orbit is isotropic then all orbits are isotropic;
see \cref{cor:sigma_constant}.)

\begin{theorem}\label{complexity_one_torsion_theorem}
Let a torus $T$ act symplectically on a closed connected symplectic manifold
$(M,\omega)$. Suppose that $\dim M = 2n$ and $\dim T \geq n-1$,
and that the orbits are isotropic.
Then every non-Hamiltonian subcircle action of $T$ 
represents a non-torsion element of $\pi_1(M)$.
\end{theorem}

We prove \cref{complexity_one_torsion_theorem} 
in \cref{proof_of_main_thm_section}.
When $M$ is symplectically aspherical, the result of 
\cref{complexity_one_torsion_theorem} also follows
from Lupton and Oprea's work~\cite[p.264]{lupton_oprea}. 
When $\dim T > n-1$, we expect that the ``isotropic orbits'' assumption 
in \cref{complexity_one_torsion_theorem} can be relaxed;
see \cref{non_isotropic_conjecture}.

\cref{complexity_one_torsion_theorem} allows us to prove 
the non-extendability of certain symplectic circle and torus actions:
\begin{corollary} \label{extension_corollary} 
In dimension $6$,
a non-Hamiltonian symplectic circle action with contractible orbits
(e.g. the free action constructed by Kotschick~\cite{kotschick_contractible_orbits},
or the actions with fixed points constructed by McDuff~\cite{mcduff_circle_actions}, Tolman~\cite{tolman_isolated_points},
and Jang--Tolman~\cite{jang_tolman}) cannot be extended to any symplectic
$T^2$ action.

In dimension $2n$, 
a non-Hamiltonian symplectic circle action with contractible orbits
cannot be extended to a symplectic $T^{n-1}$ action with isotropic orbits.
\end{corollary}

We prove \cref{extension_corollary} 
in \cref{proof_of_main_thm_section}.
As far as we know, \cref{extension_corollary} is the first
non-trivial result of non-extendability of symplectic torus actions
in these dimensions. 
For extending a Hamiltonian $T^{n-1}$ action to a Hamiltonian 
$T^{n}$ action on $M^{2n}$,
see Karshon~\cite[Prop. 5.21]{periodic_hamiltonians} for dimension $2n=4$,
and the upcoming Liu--Palmer--Tolman~\cite{extending_tall_complexity_one}
and Liu~\cite{extending_short_complexity_one} for arbitrary dimensions.
Combining results of Henigman \cite{classification_non_ham_s1}
and Pelayo \cite{pelayo-MemoirsAMS}
should yield extendability criteria for non-Hamiltonian symplectic 
circle actions to $T^2$ actions in dimension four;
see \cite[Remark 2.19]{classification_non_ham_s1}.

An important ingredient in the proof of \cref{thm:main}
is that for a symplectic $T^{n-1}$ action with isotropic orbits,
every \textit{purely non-Hamiltonian} subtorus action (see \cref{def:purely_non_ham}) acts in a
locally-free fashion (\cref{infinitesimal_splitting_lemma}).
Combining this with a result of Benoist~\cite[Corollary 3.2]{benoist},
we obtain the following
result, which has independent interest.

\begin{lemma}[Splitting Lemma] \label{splitting_lemma}
Let a torus $T$ act symplectically on a closed connected symplectic manifold
$(M,\omega)$. Suppose that $\dim M = 2n$ and $\dim T \geq n-1$,
and that the orbits are isotropic.
Then $T$ decomposes into subtori, $T = T_{\ham} \times T_c$,
such that the action of $T_{\ham}$ is Hamiltonian and the action of $T_c$
is locally-free.

Moreover, every one-parameter subgroup of $T$ whose action is Hamiltonian
is contained in $T_{\ham}$, and the action of any closed subgroup of $T$
whose intersection with $T_{\ham}$ is finite is locally-free. 
\end{lemma}

We prove \cref{splitting_lemma} in Section~\ref{splitting_lemma_section}.
When $\dim T = n$ (and the orbits are Lagrangian), 
the result of \cref{splitting_lemma} follows from a result of Benoist
for coisotropic orbits
(see ~\cite[Lemma 6.7]{benoist} or~\cite[Lemma 3.6 and Corollary 3.11]{duistermaat_pelayo}).
When the symplectic manifold satisfies the weak Lefschetz property,
the result of \cref{splitting_lemma} is true
regardless of the dimension of the torus, by a result of Ginzburg
(see~\cite[Proposition 2.2 and Proposition 4.2]{ginzburg_symplectic_actions}).
%

\begin{remark} \label{rk:cannot_decrease}
The dimension of the torus $T$ in the statements of
\cref{thm:main}, \cref{corollary_tfae}, \cref{complexity_one_torsion_theorem}, 
and \cref{splitting_lemma}
cannot be decreased to below $n-1$: as mentioned above, 
McDuff \cite{mcduff_circle_actions}
constructed a non-Hamiltonian symplectic circle action 
with fixed points on a 6-manifold. 
\end{remark}

\begin{remark}\label{remark_relation_lemma_classification}
\cref{splitting_lemma} restricts the possible local models 
that can occur in a non-Hamiltonian symplectic $T^{n-1}$ action (see also \cref{chern_class_vanishes_proposition} for another restriction). 
This should be an important ingredient for proving an equivariant 
classification theorem for non-Hamiltonian symplectic $T^{n-1}$ actions,
generalizing the four-dimensional equivariant classification 
in~\cite{classification_non_ham_s1} in the spirit of Karshon and Tolman's 
classification of tall complexity one 
spaces~\cite{centered_hamiltonians, tall_uniqueness, tall_existence}.
See also \cite[Remark 2.17]{classification_non_ham_s1}.
\end{remark}

\subsection{Free symplectic circle actions} \ 
\label{subsec:free}

Next, we discuss free symplectic circle actions with contractible orbits. 
Let $(M, \omega)$ be a closed connected symplectic manifold, 
equipped with a free symplectic circle action.
Following McDuff~\cite{mcduff_circle_actions}, if the de Rham cohomology
class $[\omega]$ is integral, then the circle action is generated by
a circle-valued momentum map $\mu \colon M \rightarrow \mathbb{R}/\mathbb{Z}$.
Currently, the only known examples of free symplectic circle actions with
contractible orbits are given by Kotschick \cite{kotschick_contractible_orbits} 
and Tolman \cite{tolman_isolated_points}, and their reduced spaces
with respect to a circle-valued momentum map are symplectomorphic 
to K3 surfaces.
It would be interesting to find new examples with different reduced spaces; see \cref{reduced_spaces_open_question} and \cref{reduced_spaces_quantitative_open_question}.

In \cref{proof_of_second_thm_section}, using long exact sequences of homotopy 
groups together with a result of 
Fern\'andez--Gray--Morgan~\cite{fernandez_gray_morgan}, we give a topological
characterization of those symplectic manifolds that can appear 
as reduced spaces of free symplectic circle actions with contractible orbits
(\cref{circle_action_characterization}). This characterization is useful for
understanding where \textit{not} to look for examples of free symplectic 
circle actions with contractible orbits (\cref{cor:coh_constraints}). 
In particular, we deduce that aspherical symplectic manifolds, and 
symplectic manifolds with second Betti number equal to one, cannot appear as
reduced spaces of free symplectic circle actions with contractible orbits 
(\cref{e:contractible_examples}). See \cref{reduced_spaces_quantitative_open_question} for a related question on the possible second Betti numbers of reduced spaces of free symplectic circle action with contractible orbits.

\subsection{Outline of the paper} \
In \cref{preliminaries_section}, we give background on real Chern classes, 
Hamiltonian actions, and non-Hamiltonian symplectic actions. 
In \cref{contractible_section}, we discuss the relation between 
``contractible orbits'' and ``null-homotopic orbit maps''.
In particular, (without a symplectic form,) 
we construct a circle action with fixed points
with some non-contractible orbits
(Example~\ref{example_fixed_non_contractible}). 
Moreover, we show that 
for a locally-free action with vanishing real Chern class, 
the orbits are non-contractible (\cref{non_free_torus_bundle_lemma}).
Furthermore, we prove that 
for Hamiltonian torus actions, all orbits are contractible 
(\cref{hamiltonian_has_contractible}).
Additionally, we prove that symplectic torus actions 
with non-isotropic orbits have non-contractible orbits
(\cref{non_isotropic_non_contractible_lemma}).
In \cref{splitting_lemma_section}, we prove the Splitting 
lemma (\cref{infinitesimal_splitting_lemma}). 
In \cref{vanishing_chern_class_section}, we prove that for every 
complementary torus to the maximal Hamiltonian torus,
the real Chern class of every level set of its momentum map vanishes 
(\cref{chern_class_vanishes_proposition}).
In \cref{proof_of_main_thm_section}, 
we prove \cref{thm:main}, \cref{complexity_one_torsion_theorem}, and \cref{extension_corollary},
using the Splitting lemma 
(\cref{infinitesimal_splitting_lemma}) and 
\cref{chern_class_vanishes_proposition}.
In \cref{proof_of_second_thm_section}, we give a topological
characterization for which spaces can appear as reduced spaces
of free symplectic circle actions with contractible orbits (see \cref{circle_action_characterization}).
In \cref{open_questions_section}, we pose several open questions related to the
topology of symplectic torus actions.
In Appendix~\ref{equivariant_formality_appendix}, we prove that
a symplectic torus action is Hamiltonian if and only if
it is equivariantly formal (\cref{cor:ham_iff_eqformal}).
In Appendix~\ref{weak_main_theorem_appendix}, 
we prove a weak version of \cref{thm:main}, 
which assumes that $\dim T > n-1$. 
The proof is based on Duistermaat and Pelayo's classification of symplectic
torus actions with coisotropic orbits~\cite{duistermaat_pelayo}. 
In Appendix~\ref{seifert_lemma_appendix}, we give a direct proof of
\cref{thm:main} in dimension four, 
using the theory of Seifert fibrations.
In Appendix~\ref{perturbation_appendix}, we describe how to perturb an
invariant symplectic form to be integral, without hurting certain
properties of the action with respect to the symplectic form.

\subsection*{Acknowledgements}
We would like to thank Tara Holm, Leonid Polterovich, Daniele Sepe, Susan Tolman, and Yoav Zimhony for
useful discussions related to this project.

Some of the discussions that contributed to this work took place in the
workshop ``Workshop on Hamiltonian Geometry and Quantization'' at the
Fields Institute for Research in Mathematical Sciences in July 2024. We
wish to thank the Fields Institute and the organizers of this event.

This research is partially funded by the United States -- Israel
Binational Science Foundation (NSF-BSF Grant 2021730), and by the
Natural Sciences and Engineering Research Council of Canada (RGPIN-2024-05798).
RH is partially supported by Leonid Polterovich's ISF-NSFC
Joint Research Program (Grant/Award Number 3231/23).

\section{Preliminaries}\label{preliminaries_section}

In this section, we collect definitions and results 
that we will need throughout the paper.

\subsection{The real Chern class}

In this subsection we define the real Chern class of a $T$-manifold $M$
when the action is locally-free, i.e., the stabilizers are discrete.
Such an $M$ will appear later in the paper as a regular level set 
of a momentum map. 
When the action is free, $M/T$ is a manifold and $M$ is a principal $T$-bundle,
and the real Chern class takes values in $H^2(M/T;\t)$,
where $\t$ is the Lie algebra of $T$.
When the action is locally-free, $M/T$ is an orbifold and $M$ is a principal
orbi-bundle over $M/T$; see \cref{rk:orbifold}.
However, for our purposes, we do not need the formalism of orbifolds.
Instead, we view the real Chern class as a class in the basic cohomology
$H^2_{\basic}(M;\t)$.

\begin{remark} \label{rk:orbifold}
For early works on orbifolds, we refer the reader to
Satake~\cite{satake_1, satake_2}, Thurston~\cite{thurston_book}, 
and Haefliger~\cite{haefliger_orbifolds}.
More modern approaches include groupoids~\cite{adem_leida_ruan} 
and stacks~\cite{lerman_stacks}. See also~\cite{karshon_orbifolds}
for a diffeological approach, with a comparison to Satake's and
Haefliger's definitions. For a comprehensive introduction to orbifolds,
see~\cite{adem_leida_ruan, orbifold_ham_s1, lerman_tolman} and references therein.
\end{remark}

Let a torus $T$, with Lie algebra $\t$, act smoothly on a closed connected 
manifold $M$, and assume that the action is locally-free. 
A differential form
$\omega \in \Omega^*(M)$ is \textbf{basic} if it is invariant
and horizontal, i.e., if 
$$ \mathcal{L}_{\xi_M}\omega = 0 \quad \text{ and } \quad 
\iota_{\xi_M} \omega = 0 $$ 
for every $\xi \in \t$,
where $\xi_M$ is the vector field on $M$ corresponding to $\xi$. 
The cochain of basic forms 
$$\cdots \ \overset{d}{\longrightarrow}\ 
\Omega_{\basic}^k(M) \ \overset{d}{\longrightarrow}\  \Omega_{\basic}^{k+1}(M)
\ \overset{d}{\longrightarrow} \ \cdots,$$ 
with the usual exterior derivative $d$,
defines the basic cohomology $H^*_{\basic}(M)$ of the $T$-manifold.
By Koszul~\cite{koszul}, the basic cohomology group $H^k_{\basic}(M)$
is isomorphic to the singular cohomology group $H^k(M/T; \mathbb{R})$
of the quotient.

A \textbf{connection one-form} for the $T$-action is a $\t$-valued one-form 
$\alpha$ on $M$ that satisfies
\begin{equation*}
    \alpha(\xi_M) = \xi \quad \text{ and } \quad
 \mathcal{L}_{\xi_M} \alpha = 0
\end{equation*}
for every $\xi \in \t$. Such a one-form exists, its
exterior derivative $d\alpha$ is a basic form, 
and the basic cohomology class
$[d\alpha]_{\basic}$ is independent of the choice of $\alpha$. 
We use this class to define the real Chern class of the $T$-action.

\begin{definition}\label{real_chern_class_definition}
Let $M$ be a locally-free $T$ manifold.
Its \textbf{real Chern class} 
is the basic cohomology class $c_1(M) := [d\alpha]_{\basic}$ 
in $H_{\basic}^2(M;\t)$, for any connection one-form $\alpha$.
\end{definition}

The real Chern class appears naturally when applying the Duistermaat-Heckman 
theorem, as it determines the variation of the Duistermaat-Heckman function 
for Hamiltonian $T$-actions; see the next subsection.

\subsection{Hamiltonian actions}
\label{subsec:ham actions}

In this subsection, we give background on Hamiltonian torus actions. 
For more details, see for example~\cite{audin_book, ana_cannas_notes, ggk_book, gs_book}.

Let a torus $T$ 
with Lie algebra 
$\t$ act symplectically and faithfully on a connected symplectic manifold 
$(M, \omega)$. 
The action is 
\textbf{Hamiltonian} if it admits a \textbf{momentum map}, that is, 
a $T$-invariant map $\mu \colon M \rightarrow \t^*$ that satisfies
\begin{equation}
    \iota_{\xi_M} \omega = -d \langle\mu, \xi\rangle
\end{equation}
for every $\xi$ in the Lie algebra $\t$, 
where $\xi_M$ is the corresponding vector field on $M$,
and $\langle\cdot, \cdot\rangle$ is the natural pairing 
between $\t$ and $\t^*$.
The \textbf{complexity} of the action is $k := \frac{1}{2}\dim(M) - \dim T$. 

When $M$ is compact and $k=0$, we obtain \textbf{symplectic toric
orbifolds}, which were classified (up to equivariant symplectomorphism)
by Delzant~\cite{delzant}.
When $M$ is compact and four dimensional and $k=1$,
these spaces were classified by Karshon~\cite{periodic_hamiltonians}.
When $k=1$ and the reduced spaces are two-dimensional, 
these spaces were classified by Karshon and 
Tolman~\cite{centered_hamiltonians, tall_uniqueness, tall_existence}.
Compactness of $M$ can be replaced by the assumption
that there exists a convex open subset $U$ of $\t^*$
that contains the image of $\mu$ 
and such that $\mu \colon M \rightarrow U$ is proper as a map to $U$.
We then call $(M,\omega,\mu,U)$ a \textbf{complexity $\mathbf k$ space}.
For classifications in this generality,
see~\cite{karshon-lerman, centered_hamiltonians, tall_uniqueness, 
tall_existence}.

The Guillemin--Sternberg--Marle local normal form 
theorem~\cite{GS_local_normal_form, marle_local_normal_form} 
describes neighbourhoods of orbits of Hamiltonian actions of compact Lie groups.
We will need it for torus actions.

\begin{definition}\label{def:ham_model}
For a torus $T$, a \textbf{Hamiltonian $\mathbf T$-model} 
is a Hamiltonian $T$-manifold that is obtained by the following construction.
Let $H$ be a closed subgroup of $T$, choose a splitting 
$\t^* \cong \h^* \times (\t/\h)^*$, and let $(V,\omega_V)$ be a symplectic 
vector space with a linear symplectic $H$ action and a quadratic momentum map
$\mu_V \colon V \to \h^*$.
Consider $T \times V \times (\t/\h)^*$ with the $H$ action
where $h \in H$ acts by (right) multiplication by $h^{-1}$ on the $T$ factor, 
by the given representation on the $V$ factor, and trivially on the last factor,
and with the $T$-action by left multiplication on the first factor.
Equip it with the two-form $\omega' \oplus \omega_V$,
where $\omega'$ the pullback under the inclusion map 
$T \times (\t/\h)^* \to T \times \t^*$ of the canonical two-form on the 
cotangent bundle of $T$.
The model $T \times_H V \times (\t/\h)^*$ is the quotient of 
$T \times V \times (\t/\h)^*$ by the $H$ action,
with the two-form that pulls back to $\omega' \oplus \omega_V$,
and with the momentum map $\mu([t,z,\nu]) = \nu + \mu_V(z)$.
\end{definition}
\begin{theorem}\label{local_normal_form_theorem}
Let a torus $T$ act on a symplectic manifold $(M, \omega)$ with momentum map 
$\mu \colon M \rightarrow \t^*$. 
Then for every $x \in M$ there exist 
a Hamiltonian $T$-model $T \times_H V \times (\t/\h)^*$
and an equivariant symplectomorphism 
from an invariant neighbourhood of $T \cdot x$ in $M$ to a neighbourhood 
of the zero section $T \times_H \{0\} \times \{0\}$ in the Hamiltonian
$T$-model that takes $x$ to $[1,0,0]$.
\end{theorem}

Next, we introduce the Duistermaat-Heckman measure. The measure induced on $M$ 
by the volume form $\omega^n/n ! $ is called the \textbf{Liouville measure}. 
Pushing it forward to $\t^*$ by the momentum map $\mu$, we obtain the 
\textbf{Duistermaat-Heckman measure} on $\t^*$. 
The Duistermaat-Heckman measure can be expressed as Lebesgue measure on $\t^*$
times the \textbf{Duistermaat-Heckman function}
$f_{\DuHe} \colon \t^* \rightarrow \mathbb{R}_{\ge 0}$,
which has the following properties.  
The function $f_{\DuHe}$ is continuous on the momentum image; 
the momentum image decomposes into polyhedral \textit{chambers}
(connected components of the regular values of $\mu$),
divided by \textit{walls}
(connected components of $\mu(M^H)$ for 1-dimensional subtori $H \subset T$);
and on each chamber, $f_{\DuHe}$ is polynomial 
of degree at most the complexity $k$.
For more information see, e.g.,~\cite{duistermaat_heckman, goldin-holm-jeffrey};
also see~\cite{atiyah_convexity, gs_convexity, gls1, gls2, gs_birational}.

Let $y \in \t^*$ be a regular value of $\mu$. Then the level set 
$\mu^{-1}(y)$ is a smooth manifold, which is preserved under the $T$-action 
since $\mu$ is $T$-invariant.
Because $y$ is a regular value of $\mu$, 
the action of $T$ on $\mu^{-1}(y)$ is locally-free.
Let $i_y \colon \mu^{-1}(y) \rightarrow M$ be the inclusion map
of the level set. Then the two-form $i_y^*\omega$ is basic,
and it defines a basic cohomology class $[i_y^*\omega]_{\basic}$
in $H^2_{\basic}(\mu^{-1}(y))$.
We will also need to consider
the real Chern class $c_1 \in H_{\basic}^2(\mu^{-1}(y), \t)$ 
of the $T$-manifold $\mu^{-1}(y)$; see \cref{real_chern_class_definition}.

\begin{remark}
In the language of orbifolds, the quotient space $\mu^{-1}(y)/T$ 
is an orbifold~$M_y$, and there exists a unique symplectic form $\omega_y$
on $M_y$ satisfying
$        \pi^*\omega_y = i_y^*\omega$,
where $\pi \colon \mu^{-1}(y) \rightarrow \mu^{-1}(y)/T$ is the quotient map.
The symplectic orbifold $(M_y, \omega_y)$ is called the \textbf{reduced space}
at $y$, see~\cite{mardsen_weinstein}. 
The basic cohomology class $[i_y^*\omega]_{\basic}$ can be identified 
with the de Rham cohomology class $[\omega_y] \in H_{\dR}^2(\mu^{-1}(y)/T)$,
and the basic cohomology class $c_1$ can be identified 
with the de Rham cohomology class of the curvature two-form
of any connection one-form on the orbi-bundle
$\mu^{-1}(y) \rightarrow \mu^{-1}(y)/T$.
\end{remark}

Since $y$ is a regular value and $\mu$ is proper, 
there exist a neighbourhood $U$ of $y$
and a diffeomorphism $\Psi \colon \mu^{-1}(U) \to U \times \mu^{-1}(y)$
such that the following diagram commutes.
    \[
    \begin{tikzcd}[column sep=large, row sep=large]
    \mu^{-1}(U) \arrow[r, "\Psi"] \arrow[rd, "\mu"'] & U \times \mu^{-1}(y) \arrow[d, "\projection_1"] \\
    & U
    \end{tikzcd}
    \]
This allows us to identify the level set $\mu^{-1}(z)$ 
with the level set $\mu^{-1}(y)$ for each $z \in U$.  
The basic cohomology class 
$[i_z^*\omega]_{\basic} \in H^2_{\basic} (\mu^{-1}(z))$,
where $i_z \colon \mu^{-1}(z)\rightarrow M$ is the inclusion map,
then becomes a class in $H_{\basic}^2(\mu^{-1}(y))$.
Another formulation of the Duistermaat-Heckman theorem is that
\begin{equation}\label{dh_variation_equation}
    [i_z^*\omega]_{\basic} = [i_y^*\omega]_{\basic} + \langle z - y, c_1\rangle
\,.
\end{equation}
The first formulation of the Duistermaat-Heckman theorem, which claims that
$f_{\DuHe}$ is a piecewise polynomial, can be proved directly from 
Equation~\eqref{dh_variation_equation}.

\subsection{Non-Hamiltonian symplectic actions}\ 
\label{subsec:nonHam}

This subsection contains background on non-Hamiltonian symplectic 
torus actions. For more information,
see~\cite{benoist, ginzburg_symplectic_actions, mcduff_circle_actions, ratiu_ortega_book}. 
We start by recalling an interesting observation 
from~\cite{benoist, duistermaat_pelayo, ginzburg_symplectic_actions}:

\begin{lemma}  \label{lem:independent_form}
Let a torus $T$ with Lie algebra $\t$
act on a connected symplectic manifold $(M,\omega)$.  
For each $\xi \in \t$, let $\xi_M$ be the corresponding vector
field on $M$. 
Then the bilinear form
$$(\xi,\xi') \mapsto \omega_x(\xi_M, {\xi_M'})$$
on $\t$ is independent of $x \in M$.
\end{lemma}

\begin{proof}
More generally, we will allow $T$ to be any connected abelian Lie group.

We have
\begin{equation*}
\begin{aligned}
d\iota_{{\xi_M'}}\iota_{\xi_M}\omega  
 & = \mathcal{L}_{{\xi_M'}} \iota_{\xi_M}\omega 
   - \iota_{{\xi_M'}}d\iota_{\xi_M}\omega \\
    &= \iota_{\xi_M}\mathcal{L}_{{\xi_M'}}\omega 
  + \iota_{[\xi_M, {\xi_M'}]}\omega  
  - \iota_{{\xi_M'}}\mathcal{L}_{\xi_M}\omega 
  + \iota_{{\xi_M'}}\iota_{\xi_M}d\omega \\  
 &= 0,
\end{aligned}
\end{equation*}
where the first equality is by Cartan's formula, the second equality is
by the identity $[\mathcal{L}_X, \iota_Y] = \iota_{[X, Y]}$ and Cartan's
formula, and the last equality is by the $T$-invariance of $\omega$,
by $[\xi_M, {\xi_M'}] = 0$, and by $d\omega = 0$. Therefore, the
function $x \mapsto \omega_x(\xi_M, {\xi_M'})$ is constant on $M$.
\end{proof}

We deduce the following useful corollary 
(see~\cite[Lemma 2.12]{duistermaat_pelayo} or~\cite[Lemma 2.1]{benoist}):

\begin{corollary}\label{cor:sigma_constant}
Given an action of a torus $T$ 
on a connected symplectic manifold $(M,\omega)$,
the pullback 
\begin{equation}\label{sigma_definition}
 \omega_T := \rho_x^* \omega 
\end{equation}
by the orbit map $\rho_x \colon T \to M$, \ $a \mapsto a \cdot x$,
is a two-form on $T$ with constant coefficients that is 
independent of $x \in M$. In particular, if one orbit is isotropic, 
then all orbits are isotropic.
\end{corollary}

\begin{remark} \label{fixed implies isotropic}
By \cref{cor:sigma_constant}, if there exists a fixed point,
then all orbits are isotropic.
Thus, Hamiltonian torus actions have isotropic orbits.
\end{remark}

For Hamiltonian flows, we deduce the following from
\cref{lem:independent_form}:

\begin{lemma} \label{lem:ham_is_invariant}
Let a torus $T$ with Lie algebra $\t$
act on a connected symplectic manifold $(M,\omega)$, and let $\xi \in \t$.
Suppose that the vector field $\xi_M$ is generated by a Hamiltonian
function $\mu^\xi \colon M \to \R$. Then $\mu^\xi$ is $T$-invariant.
\end{lemma}

\begin{proof}
More generally, we will allow $T$ to be any connected abelian Lie group,
and will assume that the action has at least one compact orbit.

By Hamilton's equation for $\mu^\xi$,
for each $\eta \in \t$ we have
\begin{equation*}\label{eq_ham_lemma}
    d\mu^\xi(\eta_M) = -\omega(\xi_M, \eta_M)
\end{equation*}
at any point. Let $\calO$ be a compact $T$-orbit 
and $x \in \calO$ be a point where $\mu^\xi|_\calO$ attains its maximum.
Because $\eta_M|_x$ is in $T_x\calO$ and $d\mu^\xi$ vanishes on $T_x\calO$,
we have $d\mu^\xi(\eta_M) = 0$ at $x$, and therefore $\omega(\xi_M, \eta_M) = 0$ at $x$. By \cref{lem:independent_form}, $\omega(\xi_M, \eta_M) = 0$ everywhere, so $d\mu^\xi(\eta_M) = 0$ everywhere as well.
Varying $\eta$, we obtain the lemma.
\end{proof}

Let a torus $T$ with Lie algebra $\t$ act symplectically on a connected
symplectic manifold $(M, \omega)$.
Because the action is symplectic, the one-form $\iota_{\xi_M}\omega$
is closed for every $\xi$ in $\t$. 
Hence, $\int_A \iota_{\xi_M} \omega$ is well defined for $A \in H_1(M;\Z)$.
%

\begin{definition}\label{def:group_of_periods}
The \textbf{group of periods} of the $T$-action is 
$$ P := \Big\{ \xi \mapsto \int_A \iota_{\xi_M}\omega \ | \ 
A \in H_1(M;\Z) \Big\} \, \subset \, \t^*\,.$$
\end{definition}

The $T$-action is Hamiltonian if and only if $P = \{0\}$. 
In particular, for non-Hamiltonian actions, the group of periods $P$ 
is non-trivial.
This relationship can be further sharpened by relating $P$ with the
following important subalgebra.

\begin{definition}\label{maximal_hamiltonian_subalgebra}
The maximal Hamiltonian Lie subalgebra of $\t$ is 
\begin{equation*}
\t_{\ham} := \{ \xi \in \t \ | \ \xi_M \text{ is Hamiltonian} \}.
\end{equation*}
\end{definition}

Let $\t_{\ham}^0$ be the annihilator of $\t_{\ham}$ in $\t^*$, 
and let $\on{span}_{\mathbb{R}} P$ be the span of $P$ in $\t^*$. 
\begin{lemma}\label{lem:spanp}
    $\on{span}_{\mathbb{R}} P = \t_{\ham}^0$.
\end{lemma}

\begin{proof}
Identifying $\t$ with $(\t^*)^*$, we denote by $P^0$
the annihilator of $P$ in $\t$.  
By definition, for each $\xi \in \t$, we have $\xi \in P^0$ if and only if
$\int_A \iota_{\xi_M} \omega = 0$ for all $A \in H_1(M;\Z)$.
By the definition of $\t_{\ham}$, this holds exactly if $\xi \in \t_{\ham}$.
So $ (\on{span}_\R P)^0 = P^0 = \t_{\ham}$. The result follows.
%
\end{proof}

When the group of periods $P$ is discrete in $\t^*$, 
the quotient $\t^*/P$ is a \textbf{cylinder}, i.e., the product 
of a Euclidean space and a torus; see \eqref{cylinder}.
We can then define a cylinder-valued momentum map
$\mu \colon M \rightarrow \t^*/P$ that generates the torus action. 

More generally, one can consider a cylinder-valued momentum map 
with values in the quotient of $\t^*$ 
by any discrete subgroup of $\t^*$ that contains $P$.
A cylinder-valued momentum map is \textbf{tight}
if it takes values in $\t^*/P$ (see \cite[Definition 4.1]{classification_non_ham_s1}).
When $M$ is connected, a cylinder-valued momentum map is tight
if and only if its level sets are connected (see \cref{connected_level_sets_lemma}).

In the literature, it is usually assumed 
that the cohomology class $[\omega] \in H^2(M, \mathbb{R})$
is integral, that is, contained in the image
of $H^2(M, \mathbb{Z}) \hookrightarrow H^2(M, \mathbb{R})$. 
This implies that the group of periods $P$ is a subset of
the weight lattice $\t_{\mathbb{Z}}^*$; 
in particular, it is a discrete subgroup of $\t^*$. 
(The weight lattice $\t_{\mathbb{Z}}^* \subset \t^*$
is the dual lattice to the integral lattice
$\t_{\mathbb{Z}} := \ker(\exp) \subset \t$.)
In this case, there is a (not necessarily tight) cylinder-valued momentum map
with values in $\t^*/\t^*_\Z$.
For more information on cylinder-valued momentum maps,  
see~\cite{cylinder_momentum_map}, \cite{mcduff_circle_actions},
\cite[Section 5.2]{ratiu_ortega_book}, or~\cite{pelayo_circle_valued_ham}.

We can use \cref{lem:spanp} to choose a convenient basis for $\t$ 
with respect to $P$
that will be useful when using cylinder-valued momentum maps:

\begin{lemma}\label{lem:split_basis}
Let a torus $T$ act on symplectically on a compact connected symplectic 
manifold $(M, \omega)$, with isotropic orbits. 
Let $P$ be the group of periods (\cref{def:group_of_periods}), 
and let $\t_{\ham}$ be the maximal Hamiltonian Lie subalgebra 
(\cref{maximal_hamiltonian_subalgebra}). 
Assume that $P$ is discrete.
Let $m := \dim T$ and $k := \dim \t_{\ham}$. 
Furthermore, let $\t_c$ be a complementary Lie subalgebra 
to $\t_{\ham}$ in $\t$.
    
Then $P$ can be identified with a subgroup of $\t_c^*$, and we can find
a basis $\xi_1,\dots,\xi_m$ for $\t$ such that 
$\xi_1,\dots,\xi_k$ are a basis for $\t_{\ham} \cong \mathbb{R}^k$, 
and $\xi_{k+1},\dots,\xi_m$ are a basis for 
$\t_c \cong \mathbb{R}^{m-k}$, and such that the dual basis 
$\xi_1^*,\dots,\xi_m^*$ satisfies that $\xi_{k+1}^*,\dots,\xi_m^*$ 
are a basis for $P \subset \t_c^*$, and we have the identification
\begin{equation}\label{cylinder}
\t^*/P \cong 
 \t_{\ham}^* \times (\t_c^*/P) \cong \mathbb{R}^k \times\mathbb{T}^{m - k}.
\end{equation}
Additionally, if $\mu \colon M \rightarrow \t^*/P$ 
is a cylinder-valued momentum map for the $T$-action, then its compositions 
with the projections $\t^*/P \rightarrow \t_c^*/P$ 
and $\t^*/P \rightarrow \t_{\ham}^*$ define (cylinder-valued) momentum maps 
for the $\t_c$-action and $\t_{\ham}$-action, respectively.
    
\end{lemma}

\begin{proof}
    Since $\t_c$ is complementary to $\t_{\ham}$, we can identify $\t_c$ with $\t_{\ham}^0$, and therefore by \cref{lem:spanp} it can be identified with $\on{span}_{\mathbb{R}} P$. The claim then follows by choosing an arbitrary basis $\xi_1^*,\dots,\xi_k^*$ for $\t_{\ham}^*$, an arbitrary basis $\xi_{k-1}^*,\dots,\xi_m^*$ for the lattice $P$, and dualizing.
\end{proof}

Cylinder-valued momentum maps retain many of the useful properties of
ordinary momentum maps. For example, if the orbits of a $T$-action
are isotropic, then its cylinder-valued momentum map is $T$-invariant:

\begin{lemma}\label{lemma:isotropic_invariant}
Let a torus $T$ act on a connected symplectic manifold $(M,\omega)$,
with isotropic orbits. 
Assume that the group of periods $P$ is discrete.
Let $\mu \colon M \rightarrow \t^*/P$ be a cylinder-valued momentum map 
for the action. Then $\mu$ is $T$-invariant.
\end{lemma}

\begin{proof}
A $T$-action with isotropic orbits is always locally Hamiltonian 
around each orbit, so the lemma follows from \cref{lem:ham_is_invariant}.
\end{proof}

For Hamiltonian torus actions on compact connected manifolds,
the level sets of the momentum map are connected \cite{atiyah_convexity}
(also see \cite{lerman_tolman}).
A similar result holds for tight cylinder-valued momentum maps:

\begin{lemma}[\cite{benoist}]\label{connected_level_sets_lemma}
Let $T$ act symplectically on a closed connected symplectic manifold 
$(M, \omega)$, and assume that the group of periods $P$ is discrete. 
Then the level sets of a (tight) cylinder-valued momentum map 
$\mu \colon M \rightarrow \t^*/P$ are connected.
\end{lemma}

\begin{proof}
This follows from a theorem of Benoist \cite[Theorem 3.1]{benoist}.
Also see~\cite[Lemma 2.2]{classification_non_ham_s1} for a simple proof in
the case of circle-valued momentum maps.
\end{proof}

Lastly, we refer to the following non-trivial result 
of Benoist~\cite[Corollary 3.2]{benoist}.

\begin{lemma}
\label{maximal_hamiltonian_lemma}
Let $T$ act symplectically on closed connected symplectic manifold 
$(M, \omega)$,
and let $\t_{\ham} \subset \t$ be its maximal Hamiltonian Lie subalgebra
(\cref{maximal_hamiltonian_subalgebra}).
Then the subgroup $T_{\ham} := \exp(\t_{\ham}) \subset T$ 
generated by $\t_{\ham}$ is closed in $T$, and is therefore a subtorus.
In particular, $T_{\ham}$ is the unique maximal subtorus of $T$ 
whose action is Hamiltonian.
\end{lemma}

\begin{proof}
This is proved by Benoist in~\cite[Corollary 3.2]{benoist}.
See also Ginzburg's~\cite[Proposition 4.2]{ginzburg_symplectic_actions} 
for an earlier proof when $[\omega]$ is rational.
\end{proof}

\begin{remark}
Benoist's proof of \cref{maximal_hamiltonian_lemma} uses ``soft'' methods.
``Hard'' methods allow a quicker proof:
$T_{\ham} := \exp(\t_{\ham})$ 
coincides with the identity component of the intersection 
$T \cap \on{Ham}(M,\omega)$ of $T$ 
with the group $\on{Ham}(M,\omega)$ of Hamiltonian diffeomorphisms of $M$. 
The flux conjecture (see Ono~\cite{ono_flux})
implies that the intersection $T \cap \on{Ham}(M,\omega)$  
is closed in $T$, so its identity component $T_{\ham}$ is a subtorus of $T$.
(We note that the intersection $T \cap \on{Ham}(M,\omega)$ 
can be disconnected: take, for example, a circle action on $S^2 \times T^2$ 
that rotates $S^2$ with speed $1$ and $T^2$ with speed $2$.)
\end{remark}

\section{Contractible orbits}\label{contractible_section}

This section is about contractibility of orbits of a torus action.
Here is an overview of the main results that we prove.
For smooth torus actions (without a symplectic form),
``the orbit maps are null-homotopic'' is equivalent to ``some orbits 
are contractible'' (\cref{lem:contractible}),
but does not imply ``all orbits are contractible''
(Example~\ref{example_fixed_non_contractible}).
For locally-free torus actions with vanishing real Chern class,
the orbits are non-contractible (\cref{non_free_torus_bundle_lemma}).
For a Hamiltonian torus action, all orbits are contractible
(\cref{hamiltonian_has_contractible}).
For a symplectic torus action, if some orbit is non-isotropic,
then all orbits are non-contractible
(\cref{non_isotropic_non_contractible_lemma}).

For a related discussion 
on the contractibility of orbits of smooth and Hamiltonian torus actions,
also see~\cite[Remark 1.8 and the paragraph after Corollary 6.8]{oprea_walsh}.

\begin{lemma} \label{lem:contractible}
Let a torus $T$ act (faithfully) on a (nonempty, connected) manifold $M$.
Then the following are equivalent.
\begin{enumerate}
\item \label{some-contr} Some orbits are contractible;
\item \label{some-null} Some orbit maps are null-homotopic;
\item \label{all-null} All orbit maps are null-homotopic;
\item \label{free-contr} All free orbits are contractible.
\end{enumerate}
\end{lemma}

\begin{proof}
For any $x \in M$, 
because its orbit map $\rho_x \colon T \to M$ splits as a 
composition $T \longrightarrow T \cdot x \xrightarrow{\text{inclusion}} M$,
if the orbit $T \cdot x$ is contractible (which means that its inclusion map 
is null-homotopic), then the orbit map $\rho_x$ is null-homotopic.
Thus, since $M$ is nonempty, \eqref{some-contr}$\implies$\eqref{some-null}.

Because $M$ is connected, any two orbit maps are homotopic.
Thus, \eqref{some-null}$\implies$\eqref{all-null}. 

For a free orbit, the orbit map $\rho_x$ is null-homotopic
if and only if the orbit is contractible.
Thus, \eqref{all-null} $\implies$\eqref{free-contr}.

Because $M$ is connected and the action is faithful,
the union of the free orbits is dense
(as a consequence of Koszul's slice theorem).
Because $M$ is nonempty, there exist free orbits. 
Thus, \eqref{free-contr}$\implies$\eqref{some-contr}.
\end{proof}

\begin{figure}[ht]
    \centering
    \begin{tikzpicture}[
        double distance=1.2pt,
        line width=0.4pt,
        node distance=3.5cm and 4.5cm 
    ]
        \node (formal) at (0,0) {\shortstack{equivariantly formal\\(e.g., Hamiltonian)}};
        \node (fixed) at (12,0) {\shortstack{$\exists$ fixed\\points}};
        
        \node (all_orb) at (0,-3) {\shortstack{all orbits are\\contractible}};
        \node (free_orb) at (6,-3) {\shortstack{free orbits are\\contractible}};
        \node (some_orb) at (12,-3) {\shortstack{some orbits are\\contractible}};
        
        \node (all_maps) at (6,-6) {\shortstack{all orbit maps\\nullhomotopic}};
        \node (some_maps) at (12,-6) {\shortstack{some orbit maps\\nullhomotopic}};

        \draw[->, double] (formal) -- node[above, font=\small] {Ginzburg--Guillemin--Karshon \cite[Prop. 3.24]{ggk_abstract}} (fixed);
        \draw[->, double] (formal) -- node[left, font=\small] {\cref{hamiltonian_has_contractible}} (all_orb);
        \draw[->, double] (fixed) -- (some_orb);
        
        \draw[->, double] (all_orb) -- (free_orb);
        \draw[->, double] (free_orb) -- node[above, font=\small] {\cref{lem:contractible}} (some_orb);
        \draw[->, double] (some_orb) -- node[right, font=\small] {\cref{lem:contractible}} (some_maps);
        \draw[->, double] (some_maps) -- node[below, font=\small] {\cref{lem:contractible}} (all_maps);
        \draw[->, double] (all_maps) -- node[left, font=\small] {\cref{lem:contractible}} (free_orb);

    \end{tikzpicture}
    \caption{Summary of logical implications between topological properties of a smooth torus action.}
    \label{fig:orbit_implications}
\end{figure}
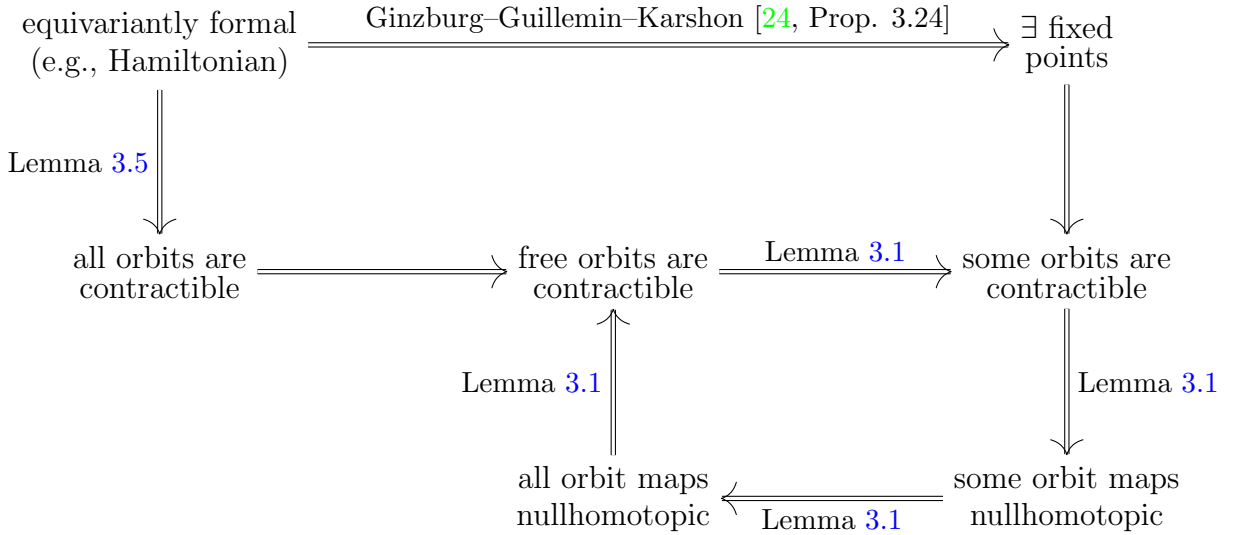

\begin{remark}
    All of the arrows in Figure~\ref{fig:orbit_implications}, that are not marked with ``\cref{lem:contractible}'', are not ``if and only if''s:
    \begin{enumerate}
        \item The standard action of $S^1$ on $S^3$ has no fixed points, but all of its orbits are contractible. Hence, ``some orbits are contractible'' does not imply ``$\exists$ fixed points'' in general. Similarly, ``all orbits are contractible'' does not imply ``equivariantly formal'' in general, because this would also imply ``$\exists$ fixed points''.
        \item Example~\ref{example_fixed_non_contractible} describes an $S^1$ action with fixed points (and with contractible free orbits), but such that not all of the orbits are contractible. Hence, ``free orbits are contractible'' does not imply ``all orbits are contractible'' in general. Similarly, ``$\exists$ fixed points'' does not imply ``equivariantly formal'' because that would also imply ``all orbits are contractible''.
    \end{enumerate}
\end{remark}
In \cite[Remark 1.8]{oprea_walsh}, Oprea and Walsh describe a non-faithful
circle action with null-homotopic orbit maps and non-contractible orbits.
In the following, we give an example of a \textbf{faithful} circle action
with null-homotopic orbit maps, but such that some of its orbits are
non-contractible. The example is constructed by taking an equivariant
connected sum between $S^2\times\mathbb{T}^2$ and $S^4$, around a free
orbit. This means extracting a neighbourhood $D^3 \times S^1$ of a free
orbit in each of the manifolds, and equivariantly gluing the boundaries,
where each boundary is diffeomorphic to $S^2 \times S^1$. For equivariant
connected sums, see also~\cite{grossberg_karshon_equivariant}.
\begin{example}\label{example_fixed_non_contractible}
The following is an example of a circle action, with fixed points, 
such that some of its orbits are non-contractible. Let $S^1$ act 
on $S^2 \times \mathbb{T}^2$ by rotating $S^2 \subset \mathbb{R}^3$ 
around the $z$ axis once, and rotating 
$\mathbb{T}^2 := \mathbb{R}^2/\mathbb{Z}^2$ around one of its coordinates 
twice. Explicitly, in cylindrical coordinates $h, \theta$ on $S^2$, 
and coordinates $p, q \mod 1$ on $\mathbb{T}^2$, the action can be described 
by
\begin{equation*}
        e^{it} \cdot (h, \theta, p, q) = (h, \theta + t, p + t/\pi, q),
\end{equation*}
where $e^{it}$ is the coordinate of $S^1 \subset \mathbb{C}$. This action 
is not free since every point that projects to the north or south pole 
on the sphere is stabilized by the subgroup 
$\mathbb{Z}_2 \subset S^1$. Let $x_0 := (h_0, \theta_0, p_0, q_0)$ 
be a point in $S^2 \times \mathbb{T}^2$. 
The fundamental group $\pi_1(S^2 \times \mathbb{T}^2, x_0)$
can be identified with $\mathbb{Z}^2$, with the generators $[\gamma_q]$
and $[\gamma_p]$, 
where $\gamma_q, \gamma_p \colon [0, 2\pi] \rightarrow S^2 \times \mathbb{T}^2$ 
are given by
\begin{equation*}
        \gamma_p(t) = (h_0, \theta_0, p_0 + t/2\pi, q_0), \quad \gamma_q(t) = (h_0, \theta_0, p_0, q_0 + t/2\pi).
\end{equation*}
Every free $S^1$-orbit represents the class $2[\gamma_p]$, and every
non-free $S^1$-orbit with stabilizer $\mathbb{Z}_2$ represents the class
$[\gamma_p]$.

Let $S^1$ act on $S^4 \subset \mathbb{R}^3 \times \mathbb{R}^2$ by rotating 
the $\mathbb{R}^2$ component. By performing an equivariant connected sum 
between $S^2 \times \mathbb{T}^2$ and $S^4$, along a free orbit, and away from 
$x_0$, we obtain a four-dimensional orientable $S^1$-manifold $M$ which has 
fixed points. By the Seifert-Van Kampen theorem, the fundamental group 
$\pi_1(M, x_0)$ of $M$ is isomorphic to the fundamental group of 
$S^2 \times \mathbb{T}^2$, with the extra relation that the free $S^1$-orbits 
are contractible, i.e., that $2[\gamma_p] = 0$. This is true since $S^4$ 
is simply-connected, together with the fact that for every four-manifold $N$, 
the fundamental group $\pi_1(N)$ does not change if we remove an embedded 
$D^3 \times S^1$, i.e., $\pi_1(N) \cong \pi_1(N\setminus(D^3 \times S^1))$. 
In particular, $\pi_1(M, x)$ is isomorphic to 
$\mathbb{Z}_2 \times \mathbb{Z}$, and every non-free $S^1$-orbit 
of $S^2\times \mathbb{T}^2$ represents the class $(g , 0)$ 
in $\pi_1(M, x) \cong \mathbb{Z}_2 \times \mathbb{Z}$, where $g$ is the 
generator of $\mathbb{Z}_2$. Hence, the non-free $S^1$-orbits are 
non-contractible in $M$.
\end{example}

Next, we prove \cref{non_free_torus_bundle_lemma}, which is needed for the proof of \cref{complexity_one_torsion_theorem} in \cref{proof_of_main_thm_section}. The lemma states that the orbits of a locally-free torus action, and their iterates, are non-contractible provided that the real Chern class of the action vanishes.
\begin{lemma}\label{non_free_torus_bundle_lemma}
    Let $T$ act faithfully on a closed connected smooth manifold $M$. Assume that the action is locally-free, and that the real Chern class $c_1$, given by Definition~\ref{real_chern_class_definition}, vanishes in $H_{\basic}^2(M, \t)$. Then for every $x \in M$, the orbit $T \cdot x$ is non-contractible in $M$. Furthermore, for every subcircle $C \subset T$, the orbit $C \cdot x$ and all of its iterates are non-contractible in $M$.
\end{lemma}

\begin{proof}
    Choose a connection one-form $\alpha \in \Omega^1(M, \t)$ for the $T$-action, and let $c_1 := [d\alpha]_{\basic} \in H^2_{\basic}(M, \t)$ be the real Chern class of the $T$-action, given by \cref{real_chern_class_definition}. By assumption, $[d\alpha]_{\basic} = c_1 = 0$, so there exists a basic one-form $\gamma \in \Omega_{\basic}^1(M, \t)$ such that $d\alpha = d\gamma$. We define $\t$-valued one-form $\tilde \alpha$ on $M$ by
    \begin{equation*}
        \tilde \alpha := \alpha - \gamma,
    \end{equation*}
    which is closed since
    $$d \tilde \alpha = d\alpha - d\gamma = d\alpha - d\alpha = 0.$$
    Moreover, since $\gamma$ is basic, the one-form $\tilde \alpha$ is also a connection one-form.

    Let $x$ be a point in $M$, and let $\rho_x \colon T \rightarrow M$ be its orbit map defined by $t \mapsto t \cdot x$. Because $\tilde \alpha$ is a connection one-form, its pullback $\rho_x^*{\tilde \alpha}$ to the torus $T$ is exactly the tautological one-form. In particular, the integral of $\rho_x^*\tilde \alpha$ over every subcircle of $T$ is non-zero, and therefore the induced map on singular homology ${(\rho_x)}_* \colon H_1(T, \mathbb{Z}) \rightarrow H_1(M, \mathbb{Z})$ is injective, and the lemma follows.
\end{proof}
When $\dim M = 3$ and $\dim T = 1$, \cref{non_free_torus_bundle_lemma} is equivalent to the claim that the orbits of a Seifert manifold are of infinite order in the fundamental group, whenever the Seifert Euler number (\cref{equation_euler_number}) vanishes. This was proved by Orlik in \cite[Section 5.3]{orlik_book} using the description of the fundamental group of a Seifert manifold (Equation~\eqref{fund_group}). In Appendix~\ref{seifert_lemma_appendix}, we describe the proof of the main theorems using the language of Seifert manifolds. For convenience, we also reprove the theorem of Orlik in~\cref{seifert_lemma}.

For the rest of this section, we will focus on symplectic actions.
For Hamiltonian actions, there are no examples such as \cref{example_fixed_non_contractible}: the orbits are always 
contractible. This resolves a question of Oprea and Walsh~\cite{oprea_walsh} 
in the case of abelian actions, and proves that Corollaries~6.7 and~6.8 
in their paper hold for all Hamiltonian torus actions.

\begin{lemma}\label{hamiltonian_has_contractible}
Let $M$ be a compact equivariantly formal $T$-manifold (e.g., a closed 
Hamiltonian $T$-manifold by \cref{cor:ham_iff_eqformal}). Then the orbits of the $T$-action are contractible.
\end{lemma}

\begin{proof}
Let $x$ be a point in $M$. We wish to show that $T \cdot x$ is contractible 
in $M$ as a submanifold. Let $H$ be the stabilizer of $x$, let 
$\rho_x \colon T \rightarrow M$ be the orbit map of $x$, and let 
$\overline{\rho_x} \colon  T/H \rightarrow M$ be the induced map. Then $T \cdot x$ 
is contractible if and only if $\overline{\rho_x}$ is null-homotopic. 
Let $M^H$ be the set of points that are fixed by $H$, and let $Y \subset M^H$
be the connected component of $x$ in $M^H$.

Since $M$ is equivariantly formal, the intersection of $Y$ and the fixed
point set $M^T$ is non-trivial (see~\cite[Proposition 3.24]{ggk_abstract}
or~\cite[Proposition C.28]{ggk_book}). Let $\gamma \colon [0, 1] \rightarrow Y$
be a path starting at $x$ and ending in $M^T$. Then the action of $T/H$ on
this path defines a homotopy of the map 
$\overline{\rho_x} \colon T/H \rightarrow M$ to a point, and thus the orbit 
$T \cdot x$ is contractible in $M$.
\end{proof}

The conclusion of \cref{hamiltonian_has_contractible} is not true for
non-abelian Hamiltonian actions. For instance, the standard action of $SO(3)$
on $S^2$ has one orbit, which is non-contractible.

\begin{remark}
It is natural to ask whether there exists a symplectic circle action such
that some but not all of its orbits are contractible;
see \cref{fixed_non_contractible_question} and \cref{some_non_contractible_question}.
The $S^1$-manifold of Example~\ref{example_fixed_non_contractible}
does not give such an example since it does not admit an invariant
symplectic form. Indeed, since it has
non-contractible orbits, then by \cref{hamiltonian_has_contractible}
the action cannot be Hamiltonian, and since it has fixed points,
then by McDuff~\cite{mcduff_circle_actions} it cannot be symplectic
non-Hamiltonian.
\end{remark}

For symplectic torus actions with at least one non-isotropic orbits,
all of the orbits must be non-contractible:

\begin{lemma}\label{non_isotropic_non_contractible_lemma}
Let a torus $T$ act symplectically on a closed connected symplectic
manifold $(M, \omega)$. Assume that there exists at least one
non-isotropic orbit. Then all of its orbits are non-contractible.
\end{lemma}
\begin{proof}
    By \cref{cor:sigma_constant}, all of the orbits are non-isotropic.
    Let $x$ be a point in $M$, and let $\rho_x:T \rightarrow M$ be its orbit map.
    Since $T \cdot x$ is non-isotropic, the pullback $\rho_x^* \omega$ is a non-zero $T$-invariant two-form on $T$, and thus its cohomology class $[\rho_x^* \omega]$ on $H^2(T, \mathbb{R})$ is non-zero, so $\rho_x$ cannot be nullhomotopic,
    and in particular the orbit $T \cdot x$ is non-contractible.
\end{proof}


\section{The Splitting Lemma} \ 
\label{splitting_lemma_section}
In this section, we prove \cref{infinitesimal_splitting_lemma},
which is an infinitesimal version of the Splitting Lemma
(\cref{splitting_lemma}). We then deduce \cref{splitting_lemma}. 
First, we give the following definition.

\begin{definition}\label{def:purely_non_ham}
Let a torus $T$ with Lie algebra $\t$ act on a symplectic manifold 
$(M, \omega)$.
The action is \textbf{purely non-Hamiltonian} if for every 
$\xi \in \t \setminus \{0\}$, 
the corresponding vector field $\xi_M$ is non-Hamiltonian.
\end{definition}
In fact, a torus action is purely non-Hamiltonian if and only if 
its maximal Hamiltonian Lie subalgebra $\t_{\ham}$ 
(\cref{maximal_hamiltonian_subalgebra})
is equal to $\{0\}$, if and only if it is cohomologically-free
in the sense of Ginzburg (see~\cite[Section 2]{ginzburg_symplectic_actions}). See also Lalonde--McDuff--Polterovich \cite[Corollary 2.2]{lalonde_mcduff_polterovich}.

\begin{lemma}[Infinitesimal Splitting Lemma]
\label{infinitesimal_splitting_lemma}
Let a torus $T$ act symplectically on a closed connected symplectic manifold
$(M,\omega)$. Suppose that $\dim M = 2n$ and $\dim T \geq n-1$,
and that the orbits are isotropic.
Let $\t_{\ham}$ be the maximal Hamiltonian Lie subalgebra
(\cref{maximal_hamiltonian_subalgebra}).

Then for every point in $M$ with stabilizer $H$ with Lie subalgebra $\h$,
we have $\h \subset \t_{\ham}$.
Moreover, every purely non-Hamiltonian subtorus action is locally-free.
\end{lemma}

We will prove \cref{infinitesimal_splitting_lemma} in a moment.
The proof is a higher-dimensional analogue of McDuff's proof 
that in dimension four a non-Hamiltonian symplectic circle action cannot have 
fixed points (\cite[Proposition 2]{mcduff_circle_actions}).
It is based on examining the Duistermaat-Heckman measure with respect 
to a cylinder-valued momentum map, and it uses 
the log-concavity of the Duistermaat-Heckman measure for complexity one actions (see~\cite{complexity_one_dh_concavity, graham, pelayo_lin}).
Note that for torus actions of higher complexity, the Duistermaat-Heckman measure is not necessarily
log-concave (see~\cite{karshon_dh} and~\cite{lin_dh}).
\begin{proof}[Proof of \cref{infinitesimal_splitting_lemma}]
    Let $K$ be a subtorus of $T$, whose action is purely non-Hamiltonian. We will show that the action of $K$ is locally-free. Without loss of generality,
    we will assume that $\dim T = n-1$; the case $\dim T \ge n$ then follows easily by restricting to subtori.

    It is enough to show that the action of every subcircle of $K$ is locally-free, to deduce that the action of $K$ is locally-free. Let $C \subset K$ be a subcircle. By assumption, $\k \cap \t_{\ham} = \{0\}$, hence the action of $C$ is non-Hamiltonian. By \cref{perturbation_lemma}, we can replace $\omega$ with an integral $T$-invariant symplectic form, such that $C$ still acts in a non-Hamiltonian fashion. By abuse of notation, we denote this integral symplectic form by $\omega$, and its corresponding maximal Hamiltonian subalgebra by $\t_{\ham}$. We choose a subtorus $T_c$ which contains $C$ as a subcircle, and such that its Lie subalgebra $\t_c$ is complementary to $\t_{\ham}$. We will show that the $T_c$-action is locally-free, and therefore so is the action of $C$, and deduce that the $K$-action is locally-free.

    Let $\t_{\ham}^*, \t_c^*$ be the duals of $\t_{\ham}, \t_c$, respectively. Since $\omega$ is integral, then $P$ is discrete, and we can choose a cylinder-valued momentum map $\mu \colon M \rightarrow \t^*/P$. Applying \cref{lem:split_basis}, we have the identification
    \begin{equation*}
        \t^*/P \cong \t_{\ham}^* \times (\t_c^*/P) \cong \mathbb{R}^k \times\mathbb{T}^{n-k-1}.
    \end{equation*}
    By assumption, the orbits are isotropic, and therefore the momentum map is $T$-invariant (see \cref{lemma:isotropic_invariant}). Furthermore, by \cref{connected_level_sets_lemma}, the level sets of $\mu$ are connected.
    
    Let $f \colon \t^*/P \rightarrow \mathbb{R}_{\ge 0}$ be the Duistermaat--Heckman density function, let $y$ be a point in $\t^*/P$, and let $\xi \in \t^*$ be an element in the dual of the Lie algebra. We define $\partial_{\xi} f(y_+)$ and $\partial_{\xi} f(y_-)$ by
    \begin{equation*}
        \partial_{\xi} f(y_+) = \lim_{t \rightarrow 0_+} \frac{f(y + t\xi) - f(y)}{t},\quad
        \partial_{\xi} f(y_-) = \lim_{t \rightarrow 0_-} \frac{f(y + t\xi) - f(y)}{t}.
    \end{equation*}
    In~\cite[Section 3]{graham}, which relies on a formula of Guillemin--Lerman--Sternberg given in \cite[Theorem 5.2]{gls1}, Graham shows that for every $y$ in a codimension one interior wall of the momentum image, and every $\xi \in \t^*$ which is not parallel to the interior wall, we have $\partial_{\xi} f(y_{+}) \le \partial_{\xi} f(y_{-})$. Graham's argument in fact proves that when the complexity of the action is one, then $\partial_{\xi} f(y_{+}) < \partial_{\xi} f(y_{-})$.
    
    Furthermore, for complexity one actions, it follows from the Duistermaat-Heckman theorem that the differential of $f$ is constant on each chamber. Even though our action is not Hamiltonian, the above facts are still true around the preimage of a small neighbourhood of any point in $\t^*/P$.
    
    Let $\xi \ne 0$ be an element of $P \subset \t^*$. We claim that $\xi$ is parallel to any codimension one interior wall. Otherwise, let $y$ be a point in $\t^*/P$ on a codimension one interior wall, such that $\xi$ is not parallel to the interior wall. By Graham's argument, we have a strict inequality
    \begin{equation}\label{strict_inequality_eq}
        \partial_{\xi} f(y_{+}) < \partial_{\xi} f(y_{-}).
    \end{equation}
    However, moving along the ray $y + t\xi$, since $\xi$ is in $P$, we reach $y$ again when $t = 1$, and the derivative either stayed the same (if we didn't cross any other wall), or decreased even more (if we did), so we have $\partial_{\xi} f(y_{-}) \le \partial_{\xi} f(y_{+})$, which contradicts Equation~\eqref{strict_inequality_eq}.

    Because $(\mathbb{R}\cdot P) \setminus \{0\}$ is dense in the dual $\t_c^*$ of the Lie algebra of $T_c$, this dual is contained in every codimension one interior wall.
    
    Finally, we claim that the $T_c$-action is locally-free. Otherwise, let $x$ be a point in $M$ such that the $T_c$-stabilizer has dimension at least one. Then by the local normal form theorem (\cref{local_normal_form_theorem}), there is an interior wall at $x$ which does not contain the subspace $\t_c^*$. This contradicts the previous paragraph. It follows that the $T_c$-action is locally-free, hence so is the $C$-action, and this concludes the proof that the $K$-action is locally-free.

    For the first part of the lemma, let $x$ be a point in $M$ with stabilizer $H$ with Lie subalgebra $\h$.
    Assume by contradiction that $\h \not\subset \t_{\ham}$. Since $\h$ is closed, it generates a subtorus $H_0$ in $T$.
    Since $\h \not\subset \t_{\ham}$, there exists a subcircle $C$ of $H_0$ whose Lie subalgebra $\mathfrak{c}$ is not contained in $\t_{\ham}$, and since $\mathfrak{c}$ is one-dimensional, we have $\mathfrak{c} \cap \t_{\ham} = \{0\}$. Then, applying the second part of the lemma, the action of $C$ is locally-free, which is a contradiction since $C$ fixes $x$.
\end{proof}
From \cref{infinitesimal_splitting_lemma}, together with Benoist's \cref{maximal_hamiltonian_lemma}, we deduce the Splitting Lemma which we now recall.

\newtheorem*{splitting_lemma}{\cref{splitting_lemma}}
\begin{splitting_lemma}[Splitting lemma]
Let a torus $T$ act symplectically on a closed connected symplectic manifold
$(M,\omega)$. Suppose that $\dim M = 2n$ and $\dim T \geq n-1$,
and that the orbits are isotropic.
Then $T$ decomposes into subtori, $T = T_{\ham} \times T_c$,
such that the action of $T_{\ham}$ is Hamiltonian and the action of $T_c$
is locally-free.

Moreover, every one-parameter subgroup of $T$ whose action is Hamiltonian
is contained in $T_{\ham}$, and the action of any closed subgroup of $T$
whose intersection with $T_{\ham}$ is finite is locally-free. 
\end{splitting_lemma}

\begin{proof}
By \cref{maximal_hamiltonian_lemma}, the maximal Hamiltonian Lie subalgebra $\t_{\ham}$ generates a closed subgroup of $T$, hence a subtorus, which we denote by $T_{\ham}$. 
Let $T_c$ be a closed subgroup of $T$ whose intersection with $T_{\ham}$ 
is finite. Then the Lie algebras of $T_{\ham}$ and $T_c$ intersect trivially, 
so applying \cref{infinitesimal_splitting_lemma} we deduce that the action of $T_c$ is locally-free. In particular, if we choose $T_c$ to be a complementary subtorus to $T_{\ham}$, 
then we get the aforementioned splitting $T \cong T_{\ham} \times T_c$.
\end{proof}

\section{Vanishing Chern class}
\label{vanishing_chern_class_section}

In this section, we prove \cref{chern_class_vanishes_proposition}, 
which is needed for the proof of \cref{complexity_one_torsion_theorem} 
in \cref{proof_of_main_thm_section}. The proof uses the Splitting Lemma 
(\cref{infinitesimal_splitting_lemma}), and a version of Moser's trick for complexity 
one spaces established by Karshon and Tolman in~\cite{centered_hamiltonians} 
and~\cite{tall_existence}.

For the proof of \cref{chern_class_vanishes_proposition}, we need several preliminary results. If a torus $T$ acts on a manifold $M$, then $T$ naturally acts on the tangent space $TM$ by $g \cdot (x, v) = (F_g(x), dF_g(v))$, where $F_g \colon M \rightarrow M$ is the diffeomorphism corresponding to $g \in T$. A vector field $X \in \mathfrak{X}(M)$ is equivariant if its corresponding section $M \rightarrow TM$ is a $T$-equivariant map. The flow of an equivariant vector field is equivariant. For more information see~\cite{equivariant_vector_fields}. We now establish the following lemma on equivariant vector fields in local models of Hamiltonian actions.
\begin{lemma}\label{local_eq_flow_lemma}
    Let $Y := T \times_H V \times \h^0$ be a Hamiltonian $T$-model, with a momentum map $\mu \colon Y \rightarrow \t^*$, as in \cref{def:ham_model}. Let $\eta$ be an element of the annihilator $\h^0 \subset \t^*$. Then there exists an equivariant vector field $X \in \mathfrak{X}(Y)$ that satisfies $d\mu(X) = \eta$ at every point in $Y$.
\end{lemma}

\begin{proof}
    The tangent space $TY$ can be split to $T(T \times_H V) \times T(\h^0)$, and we have $T(\h^0) \cong \h^0$. Thus, the vector field given by
    \begin{equation*}
        X([t, z, \nu]) := (0, \eta) \in T_{[t, z]}(T \times_H V) \times \h^0 \cong T_{[t, z, \nu]}Y
    \end{equation*}
    is equivariant and satisfies $d\mu(X) = \eta$ at every point.
\end{proof}

Gluing together the local vector fields of \cref{local_eq_flow_lemma}, 
we deduce the following global lemma in our setting.

\begin{lemma}\label{eq_flow_lemma}
    Let a torus $T$ act on a connected symplectic manifold $(M,\omega)$,
    with isotropic orbits. Assume that $\dim M = 2n$, $\dim T = n-1$, and that $[\omega]$ is integral.
    Let $P$ be the group of periods, let $\mu:M \rightarrow \t^*/P$ be a cylinder-valued momentum map, and let $\t_{\ham} \subset \t$ be the maximal Hamiltonian Lie subalgebra, given by \cref{maximal_hamiltonian_subalgebra}.
    
    Then for every element $\eta$ in the annihilator $\t_{\ham}^0 \subset \t^*$, there exists an equivariant vector field $X \in \mathfrak{X}(M)$ that satisfies $d\mu(X) = \eta$ at every point in $M$.
\end{lemma}
\begin{proof}
    By \cref{infinitesimal_splitting_lemma}, for every Hamiltonian $T$-model $Y := T \times_H V \times \h^0$ that is associated with an orbit in $M$, we have $\t_{\ham}^0 \subset \h^0$, so $\eta$ is in $\h^0$. Hence, since $M$ is compact, it follows by \cref{local_eq_flow_lemma} that we can choose a finite cover of $M$ by invariant open subsets $U_i \subset M$, and equivariant vector fields $X_i \in \mathfrak{X}(U_i)$ satisfying $d\mu(X_i) = \eta$. 
Let $\rho_i \colon M \rightarrow \mathbb{R}$ be an invariant partition of unity for the invariant cover, and define a vector field by $X := \sum_i \rho_i X_i$. Then $X$ is equivariant and satisfies $d\mu(X) = \eta$.
\end{proof}
Next, we need a version of Moser's trick for complexity one spaces established by Karshon and Tolman in~\cite{centered_hamiltonians, tall_existence}. First, we need the following definition from~\cite{centered_hamiltonians}.
\begin{definition}[Definition 3.3 in~\cite{centered_hamiltonians}]
    Let a torus act smoothly on oriented smooth manifolds $M$ and $M'$, and let $\mu \colon M \rightarrow \t^*$ and $\mu' \colon M' \rightarrow \t^*$ be $T$-invariant maps. A \textbf{$\boldsymbol{\mu}$-$\boldsymbol{T}$-diffeomorphism} from $M$ to $M'$ is an orientation preserving $T$-equivariant diffeomorphism $F \colon M \rightarrow M'$ which satisfies $\mu' \circ F = \mu$.
\end{definition}
Most often, $M$ and $M'$ will be complexity one spaces, and $\mu$ and $\mu'$ will be their momentum maps.
\begin{theorem}[Proposition 3.3 in~\cite{centered_hamiltonians}, Proposition 3.3 in~\cite{tall_existence}]
\label{complexity_one_isotopy}
    Let $(M, \omega, \mu, U)$ and $(M', \omega', \mu', U)$ be complexity one spaces that have the same Duistermaat-Heckman measure. Then every $\mu$-$T$-diffeomorphism from $M$ to $M'$ can be smoothly isotoped to an isomorphism of $(M, \omega, \mu, U)$ and $(M', \omega', \mu', U)$, through $\mu$-$T$-diffeomorphisms.
\end{theorem}
Note that our formulation is slightly stronger than theirs. In their paper, they claim there exists a $\mu$-$T$-diffeomorphism if and only if there exists an isomorphism. However, they also construct a smooth isotopy of $\mu$-$T$-diffeomorphisms as part of their proof, which is based on equivariant Moser's trick. See also~\cite[Lemma 9.6]{classification_non_ham_s1}.

Finally, we are ready to prove \cref{chern_class_vanishes_proposition}.
\begin{proposition}\label{chern_class_vanishes_proposition}
    Let a torus $T$ act on a connected symplectic manifold $(M,\omega)$,
    with isotropic orbits. Assume that $\dim M = 2n$, $\dim T = n-1$, and that $[\omega]$ is integral.
    Let $P$ be the group of periods, let $\t_{\ham} \subset \t$ be the maximal Hamiltonian Lie subalgebra, given by \cref{maximal_hamiltonian_subalgebra}.
    Moreover, let $T_c \subset T$ be a subtorus whose Lie algebra $\t_c$ satisfies $\t_{\ham} \oplus \t_c \cong \t$, and let $\mu_c \colon M \rightarrow \t_c^*/P$ be a cylinder-valued momentum map.
    
    Then for every $y \in \t_c^*/P$, the real Chern class $c_1$ of the $T_c$-action on $\mu_c^{-1}(y)$, given by Definition~\ref{real_chern_class_definition}, vanishes in $H_{\basic}^2(\mu_c^{-1}(y), \t_c)$.
\end{proposition}
\begin{remark}
    The real Chern class $c_1$ in the statement of \cref{chern_class_vanishes_proposition} is well-defined since the $T_c$-action is locally-free by \cref{infinitesimal_splitting_lemma}.
\end{remark}

\begin{proof}[Proof of \cref{chern_class_vanishes_proposition}]
    Choose a cylinder-valued momentum map $\mu:M \rightarrow \t^*/P$ and apply \cref{lem:split_basis} to obtain a basis $\xi_1,\dots,\xi_{n-1}$ for $\t$, with dual basis $\xi_1^*,\dots,\xi_{n-1}^*$, such that $\xi_1^*,\dots,\xi_k^*$ is a basis for $\t_{\ham}^*$, and $\xi_{k+1}^*,\dots,\xi_{n-1}^*$ is a basis for the lattice $P$ (and a basis of $\t_c^*$). We have the identification
    \begin{equation}\label{identification_equation}
        \t^*/P \cong \t_{\ham}^* \times (\t_c^*/P) \cong \mathbb{R}^k \times\mathbb{T}^{n-k-1},
    \end{equation}
    and $\mu_c$ is equal (up to a constant) to the composition of the momentum map $\mu$ with the projection $\t^*/P \rightarrow \t_c^*/P$.
    
    By \cref{eq_flow_lemma}, for every $k+1 \le j \le n -1$, we can choose an equivariant vector field $X_j$ on $M$ whose equivariant flow $\Psi^j_t \colon M \rightarrow M$ satisfies
    \begin{equation*}
        \mu \circ \Psi_t^j = \mu + t\xi_j^*.
    \end{equation*}
    By the choice of basis, $\xi_j^*$ is in $P$, hence for $t=1$ we have
    \begin{equation}\label{eq_flow_gradient}
        \mu \circ \Psi_1^j = \mu.
    \end{equation}
    
    Let $y$ be a point in $\t_c^*/P$. By \cref{infinitesimal_splitting_lemma}, the action of $T_c$ is locally-free, so we can define the real Chern class $c_1 \in H_{\basic}^2(\mu_c^{-1}(y), \t_c)$ of the $T_c$-action on $\mu_c^{-1}(y)$, see \cref{real_chern_class_definition}. Furthermore, let
    $$\Psi^j_t|_y \colon \mu_c^{-1}(y) \rightarrow \mu_c^{-1}(y + t\xi_j^*)$$
    be the restriction of $\Psi^j_t$ to $\mu_c^{-1}(y)$. For every $z \in \t_c^*/P$, let $i_z:\mu_c^{-1}(z) \rightarrow M$ be the inclusion map of the level set at $z$, and let ${(i_z)}^*\omega \in \Omega^2(\mu_c^{-1}(z), \mathbb{R})$ be the pullback of $\omega$ by $i_z$. Then by the Duistermaat-Heckman theorem (see Equation~\eqref{dh_variation_equation}), we have
    \begin{equation}\label{general_duistermaat_heckman_flow_eq}
        {(\Psi^j_t|_y)}^*[{(i_{y + t\xi_j^*})}^*\omega]_{\basic} = [{(i_y)}^*\omega]_{\basic} + t\langle c_1, \xi_j^* \rangle
    \end{equation}
    for every $t \ge 0$ as long as the ray $[y, y + t\xi_h^*]$ does not cross a codimension one interior wall. Since the $T_c$-action is locally-free, there are no codimension one interior walls in the image of $\mu_c$, hence Equation~\eqref{general_duistermaat_heckman_flow_eq} is true for every $t$.
        
    By Equation~\eqref{general_duistermaat_heckman_flow_eq}, for $t = 1$, we have that ${(\Psi^j_1|_y)}^*[{(i_{y + \xi_j^*})}^*\omega]_{\basic} = [{(i_y)}^*\omega]_{\basic} + \langle c_1, \xi_j^* \rangle$. Since $\xi_j^*$ is in $P$, we have $y = y + \xi_j^*$ in $\t^*_c/P$, and so
    \begin{equation}\label{duistermaat_heckman_flow_eq}
    {(\Psi^j_1|_y)}^*[{(i_y)}^*\omega]_{\basic} = [{(i_y)}^*\omega]_{\basic} + \langle c_1, \xi_j^* \rangle.
    \end{equation}

    Let $V \subset \t_c^*/P$ be a small convex neighbourhood of $y$, and let $\pi \colon \t^*/P \rightarrow \t_c^*/P$ be the projection map, given by the identification in Equation~\eqref{identification_equation}. The preimage $W := \pi^{-1}(V)$ in $\t^*/P$ is a small convex neighbourhood of $\pi^{-1}(y)$. See \cref{fig:cylinder_diagram}. Since $W$ is convex, we can choose a local section $\t^*/P \rightarrow \t$ that lifts $W$ to a convex subset in $\t$, hence the $T$-action on the preimage $\mu^{-1}(W)$ is Hamiltonian and defines a complexity one space $(\mu^{-1}(W), \omega, \mu, W)$.
    
    \begin{figure}
    \begin{tikzpicture}
        
        \def\cylinderRadius{1.5}
        \def\cylinderHeight{5}
        \def\ellipseXScale{1}
        \def\ellipseYScale{0.3}
        \def\ellipsesDistance{7}

        \pgfmathsetmacro{\xRadius}{\cylinderRadius * \ellipseXScale}
        \pgfmathsetmacro{\yRadius}{\cylinderRadius * \ellipseYScale}

        \begin{scope}[local bounding box=cylinder]
            
            \draw[thick] (0,0) ellipse (\xRadius cm and \yRadius cm);
            
            \draw[thick] (0, \cylinderHeight) ellipse (\xRadius cm and \yRadius cm);
            
            \draw[thick] (-\xRadius, 0) -- (-\xRadius, \cylinderHeight);
            \draw[thick] (\xRadius, 0) -- (\xRadius, \cylinderHeight);
            
            \node[above=0.15cm of {0,\cylinderHeight+\yRadius}] {$t^*/P$};
            
            \def\grayLineOffset{0.6}
            \draw[green!40!black, dashed, thick] (-\grayLineOffset, -\yRadius + 0.1) -- (-\grayLineOffset, \cylinderHeight-\yRadius);
            \draw[green!40!black, dashed, thick] (\grayLineOffset, -\yRadius + 0.1) -- (\grayLineOffset, \cylinderHeight-\yRadius);
            
            \foreach \y in {-0.3, 0, 0.3, 0.6, 0.9, 1.2, 1.5, 1.8, 2.1, 2.4, 2.7, 3, 3.3, 3.6, 3.9, 4.2} {
                \draw[green!40!black] (\grayLineOffset-0.1, \y) -- (-\grayLineOffset+0.1, \y+0.2);
            }

            \node[green!40!black, right=0.15cm of {\grayLineOffset-0.2, \cylinderHeight*0.4}] {$W$};
            
            \draw[blue!70!black, thick] (0, -\yRadius) -- (0, \cylinderHeight-\yRadius);
            \node[blue!70!black, right=0.03cm of {0, 0.635*\cylinderHeight}] {$\pi^{-1}(y)$};

        \end{scope}

        \node (ellipse_right) [draw, thick, ellipse, 
                            minimum width=\xRadius*2cm, 
                            minimum height=\yRadius*2cm, 
                            at={(\ellipsesDistance, \cylinderHeight/2)}] {};
        \node[above=0.15cm of ellipse_right.north] {$t_c^*/P$};

        \draw[green!40!black, ultra thick] (\ellipsesDistance, \cylinderHeight/2) + (250:\xRadius cm and \yRadius cm) arc (250:290:\xRadius cm and \yRadius cm);
        \draw[green!40!black, ultra thick] (6.5, 2.00) -- (6.5, 2.15);
        \draw[green!40!black, ultra thick] (7.5, 2.00) -- (7.5, 2.15);
        \node[green!40!black] at (7.3, 2.4) {$V$};
        
        \node[fill=blue!70!black, shape=circle, minimum size=2mm, inner sep=0pt, at={(\ellipsesDistance, \cylinderHeight/2-\yRadius)}, label={[blue, below, yshift=-0.15cm]:$y$}] (y) {};

        \draw[-{Stealth[length=4mm, width=3mm]}, thick] (2, \cylinderHeight/2) -- (\ellipsesDistance - \xRadius - 0.2, \cylinderHeight/2) node[midway, above] {$\pi$};

    \end{tikzpicture}
    \caption{Diagram of $y$, $\pi^{-1}(y)$, and $W$, with respect to the projection~$\pi$ from the cylinder $\t^*/P\cong\mathbb{R}^k \times\mathbb{T}^{n-k-1}$ to the torus $\t_c^*/P\cong\mathbb{T}^{n-k-1}$.}
    \label{fig:cylinder_diagram}
    \end{figure}
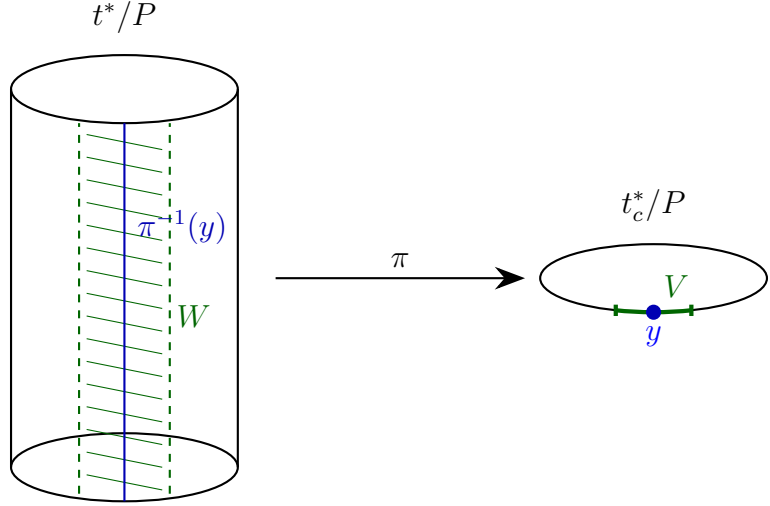

    By Equation~\eqref{eq_flow_gradient}, the equivariant diffeomorphism $\Psi_1^j:M \rightarrow M$ restricts to a $\mu$-$T$-diffeomorphism $\Psi_1^j \colon \mu^{-1}(W) \rightarrow \mu^{-1}(W)$. Hence, by \cref{complexity_one_isotopy}, we can construct an isotopy through $\mu$-$T$-diffeomorphisms from the $\mu$-$T$-diffeomorphism $\Psi_1^j$ to an isomorphism
    \begin{equation*}
        F_j \colon (\mu^{-1}(W), \omega, \mu, W) \rightarrow (\mu^{-1}(W), \omega, \mu, W)
    \end{equation*}
    of complexity one spaces. Since $F_j$ is an isomorphism, the restriction $F_j|_y \colon \mu_c^{-1}(y) \rightarrow \mu_c^{-1}(y)$ satisfies
    \begin{equation}\label{fj_symplectomorphism}
        {(F_j|_y)}^*{(i_y)}^*\omega = {(i_y)}^*\omega.
    \end{equation}
    Since $\Psi_1^j$ and $F_j$ are isotopic through $\mu$-$T$-diffeomorphisms, then the restrictions $\Psi_1^j|_y$ and $F_j|_y$ are also equivariantly isotopic. Thus, we deduce from Equation~\eqref{duistermaat_heckman_flow_eq} that
    \begin{equation*}
        {(F_j|_y)}^*[{(i_y)}^*\omega]_{\basic} = {(\Psi^j_1|_y)}^*[{(i_y)}^*\omega]_{\basic} = [{(i_y)}^*\omega]_{\basic} + \langle c_1, \xi_j^* \rangle.
    \end{equation*}
    Hence, together with Equation~\eqref{fj_symplectomorphism}, we deduce that $\langle c_1, \xi_j^* \rangle = 0$. Because this argument applies for every $\xi_j^*$ with $k + 1 \le j \le n-1$, and since $\xi_{k+1}^*,\dots,\xi_{n-1}^*$ are a basis for $\t_c^*$, we deduce that the real Chern class $c_1$ vanishes in $H_{\basic}^2(\mu_c^{-1}(y), \t_c)$, as needed.
\end{proof}
\begin{remark}
    The proof of \cref{chern_class_vanishes_proposition} fails when $\dim T < n-1$. In particular, it fails for non-Hamiltonian circle actions with fixed points in dimension 6 (as in McDuff~\cite{mcduff_circle_actions}, Tolman~\cite{tolman_isolated_points}, Jang--Tolman~\cite{jang_tolman}); these actions satisfy $\dim T = n - 2$, and their level sets have non-vanishing Chern classes. Since they have fixed points, the Splitting Lemma (\cref{infinitesimal_splitting_lemma}) is false for these examples, and so is \cref{eq_flow_lemma}. The proof also fails for Kotschick's~\cite{kotschick_contractible_orbits} free symplectic circle action with contractible orbits in dimension 6, which also satisfy $\dim T = n - 2$ and has non-vanishing Chern classes. The proof fails for Kotschick's example since we cannot apply Karshon-Tolman's version of Moser's method (\cref{complexity_one_isotopy}); their result only applies for complexity one spaces, and is false in general.
\end{remark}

\section{Proof of the main theorems}
\label{proof_of_main_thm_section}

In this section, we prove \cref{thm:main},
\cref{complexity_one_torsion_theorem}, and \cref{extension_corollary}.
We start with the proof of \cref{complexity_one_torsion_theorem},
whose statement we now recall:

\begin{quotation}
Let a torus $T$ act symplectically on a closed connected symplectic manifold
$(M,\omega)$. Suppose that $\dim M = 2n$ and $\dim T \geq n-1$,
and that the orbits are isotropic.
Then every non-Hamiltonian subcircle action of $T$ 
represents a non-torsion element of $\pi_1(M)$.
\end{quotation}

\begin{proof}[Proof of \cref{complexity_one_torsion_theorem}]
    Let $T^{n-1}$ act on $(M^{2n}, \omega)$ with isotropic orbits. Let $C \subset T$ be 
    a subcircle that acts in a non-Hamiltonian fashion. We wish to show 
    that the orbits of the $C$-action, and their iterates, are non-contractible in $M$.

    By \cref{perturbation_lemma}, there exists an integral $T$-invariant 
    symplectic form on $M$, for which the orbits of $T$ are still isotropic, 
    and such that $C$ still acts in a non-Hamiltonian fashion. 
    By replacing $\omega$ with this new symplectic form, we assume without
    loss of generality that $[\omega]$ is integral. It follows that the
    group of periods $P \subset \t^*$ of the action is a discrete
    subgroup of $\t^*$, and we can define a cylinder-valued
    momentum map
    \begin{equation*}
    \mu \colon M \rightarrow \t^*/P
    \end{equation*}
    that generates the torus action (see \S\ref{subsec:nonHam}). 
    Note that $\mu$ is $T$-invariant since the orbits are isotropic; see \cref{lemma:isotropic_invariant}.

    Let $T_c \subset T$ be a subtorus containing $C$, whose Lie
    subalgebra $\t_c$ is complementary to the maximal
    Hamiltonian Lie subalgebra $\t_{\ham}$ (given by \cref{maximal_hamiltonian_subalgebra}).
    Then by \cref{lem:split_basis}, we have the identification
    \begin{equation*}
        \t^*/P \cong \t_{\ham}^* \times (\t_c^*/P) \cong \mathbb{R}^k \times \mathbb{T}^{n-k-1},
    \end{equation*}
    and the composition of $\mu:M \rightarrow \t^*/P$ with
    the projection $\t^*/P \rightarrow \t_c^*/P$
    defines a cylinder-valued momentum map
    \begin{equation*}
        \mu_c:M \rightarrow \t_c^*/P \cong \mathbb{T}^{n-k-1}
    \end{equation*}
    for the $T_c$-action on $M$.
    
    Let $x \in M$ be some point, and let $y = \mu_c(x)$ be its image in $\t_c^*/P \cong \mathbb{T}^{n-k-1}$. 
    Since $\mu$ is $T$-invariant, it also follows that $\mu_c$ is $T_c$-invariant, so the $T_c$-orbit of $x$ is contained in the level set $\mu_c^{-1}(y)$.
    Moreover, the level set $\mu_c^{-1}(y)$ is connected by \cref{connected_level_sets_lemma}.

    By \cref{infinitesimal_splitting_lemma}, the action of $T_c$ is locally-free. Hence, the cylinder-valued momentum map $\mu_c \colon M \rightarrow \t_c^*/P$ is a submersion, and therefore by Ehresmann's theorem a locally trivial fibration. We then have the following long exact sequence of homotopy groups:
    \begin{equation*}
        \cdots \rightarrow \pi_2(\t_c^*/{P}, y) \rightarrow \pi_1(\mu_c^{-1}(y), x) \rightarrow \pi_1(M, x) \rightarrow \cdots.
    \end{equation*}
    Since the space $\t_c^*/P$ is homeomorphic to a torus, it is aspherical. It follows that the map $\pi_1(\mu_c^{-1}(y), x) \rightarrow \pi_1(M, x)$ is injective. Therefore, it is enough to show that the orbits of $C$ and their iterates are non-contractible in the level set $\mu_c^{-1}(y)$ to deduce that they are non-contractible in $M$.

    The real Chern class of the $T_c$-action on the level set $\mu_c^{-1}(y)$, given by Definition~\ref{real_chern_class_definition}, is an element of $H_{\basic}^2(\mu_c^{-1}(y), \t_c)$, which vanishes by \cref{chern_class_vanishes_proposition}.
    
    Next, we apply \cref{non_free_torus_bundle_lemma} which states that the subcircles of the orbits of a locally-free $T$-action on a closed connected manifold $N$, and their iterates, are non-contractible in $N$, provided that the real Chern class of the $T$-action vanishes. We apply this lemma to the action of $T_c$ on $\mu_c^{-1}(y)$, and deduce that the subcircles of the orbit $T_c \cdot x$, and their iterates, are non-contractible in the level set $\mu_c^{-1}(y)$. In particular, this is true for the orbits of the action of $C$, and this concludes the proof.
\end{proof}

We recall the statement of \cref{extension_corollary}:
\begin{quotation}
In dimension $6$,
a non-Hamiltonian symplectic circle action with contractible orbits
(e.g. the free action constructed by Kotschick~\cite{kotschick_contractible_orbits},
or the actions with fixed points constructed by McDuff~\cite{mcduff_circle_actions}, Tolman~\cite{tolman_isolated_points},
and Jang--Tolman~\cite{jang_tolman}) cannot be extended to any symplectic
$T^2$ action.

In dimension $2n$, 
a non-Hamiltonian symplectic circle action with contractible orbits
cannot be extended to a symplectic $T^{n-1}$ action with isotropic orbits.
\end{quotation}

\begin{proof}[Proof of \cref{extension_corollary}]
The second part of the theorem follows immediately 
from \cref{complexity_one_torsion_theorem}.

For the first part, suppose that a circle $C$ acts symplectically
on a six-dimensional manifold $(M, \omega)$, with null-homotopic orbits
(e.g., the aforementioned circle actions).
Assume by contradiction that $C \subset T^2$, 
and that there is a symplectic action of $T^2$ which extends the $C$-action.
Then by \cref{complexity_one_torsion_theorem},
the $T$-action has at least one non-isotropic orbit, 
and thus by \cref{cor:sigma_constant}, all of its orbits are non-isotropic. 
Hence, it is enough to find one isotropic $T$-orbit to get a contradiction.

Fix $x \in M$, and let $\gamma_t \colon C \to M$, for $t \in [0,1]$, 
be a smooth family of maps such that $\gamma_0$ is the orbit map 
$a \mapsto a \cdot x$ and $\gamma_1$ is the constant map $a \mapsto x$.
Choose a complementary subcircle $C'$, so that $T = C' \times C$.
Then $(a',a) \mapsto a' \cdot \gamma_t(a)$, for $(a',a) \in C' \times C$
and $t \in [0,1]$,
is a smooth homotopy between the orbit map $\rho_x \colon T \to M$
and a map $T \to M$ whose differential has rank at most $1$.
So the pullback $\rho_x^*\omega$ of $\omega$ by the orbit map $\rho_x \colon T \to M$
is a closed invariant two-form on $T$ whose cohomology class is zero, 
so it is the zero-form, hence $T \cdot x$ is isotropic, as required.
\end{proof}

The main content of \cref{thm:main} is in following proposition, which follows easily 
from \cref{complexity_one_torsion_theorem} 
and \cref{non_isotropic_non_contractible_lemma}:
\begin{proposition}\label{complexity_one_proposition}
Let a torus $T$ act symplectically on a 
closed connected symplectic manifold $(M, \omega)$.  
Suppose that $\dim M = 2n$ and $\dim T = n-1$.
If the action is non-Hamiltonian, then its orbit maps are not null-homotopic.
\end{proposition}

\begin{proof}[Proof of \cref{complexity_one_proposition}]
If the action has at least one non-isotropic orbit then by
\cref{non_isotropic_non_contractible_lemma} all of the $T$-orbits
are non-contractible. In particular, the free $T$-orbits are
non-contractible, and therefore the orbit maps are not null-homotopic
(see \cref{lem:contractible}).
    
Otherwise, all of the orbits of the action are isotropic. In this case,
by \cref{complexity_one_torsion_theorem}, 
the action has a non-Hamiltonian subcircle action, 
all of whose orbits are non-contractible. Choose a free $T$-orbit, 
which exists since the union of the free orbits is dense 
(as a consequence of Koszul's slice theorem). 
Then it is non-contractible,
since it has a non-contractible subcircle, hence its orbit map is
not null-homotopic, and so all orbit maps are not null-homotopic 
(see \cref{lem:contractible}).
\end{proof}

Finally, using \cref{complexity_one_proposition},
we prove \cref{thm:main}, whose statement we recall:
\begin{quotation}
Let a torus $T$ act symplectically on a closed connected 
symplectic manifold $(M, \omega)$.  
Suppose that $\dim M = 2n$ and $\dim T \geq n-1$.
Then the action is Hamiltonian if and only if its orbit maps are null-homotopic.
\end{quotation}

\begin{proof}[Proof of \cref{thm:main}]
As we explained in \S\ref{subsec:ham-vs-nonham}, a Hamiltonian torus
action has fixed points, so its orbit maps are null-homotopic.
For the other direction, we assume that
the action is non-Hamiltonian, and we will show that its orbit maps are
not null-homotopic. Since $\dim T \geq n-1$, we can choose a subtorus
$H \subset T$ of dimension $n-1$ which acts in a non-Hamiltonian
way. Applying \cref{complexity_one_proposition} to the
$H$-action, we deduce that the orbit maps of the $H$-action are not
null-homotopic. It follows that the orbit maps of the $T$-action cannot
be null-homotopic, as well.
\end{proof}


\section{Free symplectic circle actions}
\label{proof_of_second_thm_section}

In this section, we discuss free symplectic circle actions with contractible orbits. 
Let $(M, \omega)$ be a closed connected symplectic manifold, 
equipped with a free symplectic circle action with generating vector field 
$X \in \mathfrak{X}(M)$. 
Suppose that the de Rham cohomology class $[\omega]$ is integral, that is, in the image of $H^2(M, \mathbb{Z}) \rightarrow H^2(M, \mathbb{R}) \cong H_{\dR}^2(M)$.
Then the \textbf{group of periods}
\begin{equation*}
    P := \left \{\int_A \iota_X \omega \mid A \in H_1(M, \mathbb{Z})\right \}
\end{equation*}
is a discrete subgroup of $\mathbb{R}$, 
and the circle action is generated by a circle-valued momentum map 
$\mu \colon M \rightarrow \mathbb{R}/P$.
The orbital circle-valued momentum map 
$\overline{\mu} \colon M/S^1 \rightarrow \mathbb{R}/P$ 
is a fiber bundle whose fibers are the reduced spaces with respect to $\mu$. 
The main aim of this section is to give a topological
characterization of which spaces can appear as reduced spaces of free 
symplectic circle actions with contractible orbits (see \cref{circle_action_characterization}).
For the proof, we need the following lemma and theorem.
\begin{lemma}\label{lemma_contractible_free_orbits_homotopy}
Let $S^1$ act freely and symplectically on a closed connected symplectic
manifold $(M, \omega)$, with a discrete group of periods $P$, and let
$\mu \colon M \rightarrow \mathbb{R}/P$ be a circle-valued momentum
map for the action. Moreover, let $y$ be a point in $\mathbb{R}/P$, and
let $c_1 \in H_{\dR}^2(\mu^{-1}(y)/S^1, \mathbb{R})$ be the real Chern
class of the principal $S^1$-bundle $\pi \colon \mu^{-1}(y) \rightarrow
\mu^{-1}(y)/S^1$ (here we use Chern classes in the usual sense; we don't use basic cohomology).
Then the orbits of the action are contractible if and
only if $\langle \pi_2(\mu^{-1}(y)/S^1), c_1 \rangle = \mathbb{Z}$.
\end{lemma}

\begin{proof}
Let $x$ be some point in $\mu^{-1}(y)$, let $E := \mu^{-1}(y)$ be its
level set, and let $B := \mu^{-1}(y)/S^1$ be its quotient space. Following
the proof of \cref{complexity_one_proposition}, since
the action is free, the map $\mu$ is a submersion, and therefore
by Ehresmann's theorem a locally trivial fibration, and we have the
following long exact sequence:
\begin{equation*}
        \cdots \rightarrow \pi_2(\mathbb{R}/P, y) \rightarrow \pi_1(E, x) \rightarrow \pi_1(M, x) \rightarrow \cdots
\end{equation*}
Because $\mathbb{R}/P$ is homeomorphic to a circle, its second homotopy
group is trivial, hence the map $\pi_1(E, x) \rightarrow \pi_1(M, x)$ 
is injective. It follows that the orbits of the circle action are
contractible in $M$ if and only if they are contractible in the level set $E$.

Since the $S^1$ action on the level set $E$ is free, the quotient map
$\pi \colon E \rightarrow B$ defines a principal $S^1$-bundle, and we
have the long exact sequence:
\begin{equation*}
        \cdots \rightarrow \pi_2(B, \pi(x)) \rightarrow \pi_1(S^1, e) \rightarrow \pi_1(E, x) \rightarrow \cdots
\end{equation*}
Thus, the orbits of the $S^1$ action are contractible in $E$ if and only if 
the map $\pi_1(S^1, x) \rightarrow \pi_1(E, x)$ is trivial, and this happens 
if and only if the map $\pi_2(B, y) \rightarrow \pi_1(S^1, x)$ is surjective. 
This is true if and only if there exists a map $f \colon S^2 \rightarrow B$ 
such that the pullback bundle $f^*E \rightarrow S^2$ has Euler number one. 
Because the Euler number and the Chern number of $f^*E \rightarrow S^2$ agree, 
the Euler number is one precisely when $f^*c_1 = [S^2]$, 
where $c_1 \in H_{\dR}^2(B, \mathbb{R})$ is the Chern class of the 
principal $S^1$-bundle $\pi \colon E \rightarrow B$, and $[S^2]$ is the 
positive generator of $H_2(S^2, \mathbb{Z})$. This concludes the proof.
\end{proof}

\begin{theorem}[Reformulation of Theorem 14 in~\cite{fernandez_gray_morgan}]
\label{theorem_gray_morgan}
Let $(N^{2n}, \Omega)$ be a closed connected symplectic manifold. 
Then the following are equivalent:
\begin{enumerate}
\item\label{grey_prop_one} There exists a closed connected symplectic manifold with a free symplectic circle action, generated by a circle-valued Hamiltonian, and with some reduced space symplectomorphic to $(N, \Omega)$.
\item\label{grey_prop_two} There exists a cohomology class $c \in H^2(N, \mathbb{Z})$, and a family of symplectic forms $\Omega_t$ on $N$, satisfying the following properties:
        \begin{enumerate}
            \item $\Omega_0 = \Omega$.
            \item $[\Omega_t] = [\Omega] + tc$.
            \item There exists an orientation-preserving diffeomorphism $\Psi \colon N \rightarrow N$, such that $\Psi^*c = c$ and $\Psi^*\Omega_t = \Omega_{t - 1}$ for every $t$.
        \end{enumerate}
\end{enumerate}
\end{theorem}
\begin{theorem}\label{circle_action_characterization}
Let $(N^{2n}, \Omega)$ be a closed connected symplectic manifold. Then
the following are equivalent:
\begin{enumerate}
\item 
There exists a closed connected symplectic manifold with a free symplectic 
circle action with contractible orbits, generated by a circle-valued 
Hamiltonian, with some reduced space symplectomorphic to $(N, \Omega)$.

\item\label{second_property_in_theorem} 
There exists a cohomology class $c \in H^2(N, \mathbb{Z})$, and a family of 
symplectic forms $(\Omega_t)_{t \in \R}$ on $N$, with the following properties:
\begin{enumerate}
\item $\Omega_0 = \Omega$.
\item $[\Omega_t] = [\Omega] + tc_\R$,
where $c_\R$ is the image of $c$
under the homomorphism $H^2(N;\Z) \to H^2_{\dR}(N)$.
\item\label{part_c} There exists an orientation-preserving diffeomorphism 
$\Psi \colon N \rightarrow N$, such that $\Psi^*c = c$ and 
$\Psi^*\Omega_t = \Omega_{t - 1}$ for every $t$.
\item $\langle c, \pi_2(N)\rangle = \mathbb{Z}$.
\end{enumerate}
\end{enumerate}
\end{theorem}

The difference between the first property of \cref{theorem_gray_morgan} 
and \cref{circle_action_characterization} is the new requirement
that the orbits are contractible, and the difference between
the second property of \cref{theorem_gray_morgan} and
\cref{circle_action_characterization} is the new requirement that
$\langle c, \pi_2(N)\rangle = \mathbb{Z}$.

\begin{proof}[Proof of \cref{circle_action_characterization}]
    Following the proof of \cref{theorem_gray_morgan} in~\cite{fernandez_gray_morgan}, the relation between property~\eqref{grey_prop_one} and property~\eqref{grey_prop_two} in \cref{theorem_gray_morgan} is as follows. Let $N$, $c$, and $\Psi$ be as in property~\eqref{grey_prop_two}. Then taking the principal $S^1$-bundle $E$ over $N$ with Chern class $c$, we can lift $\Psi$ to a diffeomorphism $\tilde \Psi \colon E \rightarrow E$, and the quotient space
    \begin{equation*}
        \bigslant{E \times \mathbb{R}}{(x, t)\sim(\tilde \Psi(x), t + 1)}
    \end{equation*}
    can be endowed with an invariant symplectic form precisely if we have $\Omega_t$ as above.

    Thus, in the language of \cref{theorem_gray_morgan}, $c$ is identified with the Chern class of each level set, so \cref{lemma_contractible_free_orbits_homotopy} implies that the orbits are contractible if and only if $\langle c, \pi_2(N)\rangle = \mathbb{Z}$, and this finishes the proof.
\end{proof}
A K3 surface is the only known symplectic four-manifold to satisfy
Property~\eqref{second_property_in_theorem} of
\cref{circle_action_characterization}.
Kotschick \cite{kotschick_contractible_orbits} 
used non-trivial facts about the geometry of 
K3 surfaces for constructing the diffeomorphism $\Psi$ 
in Property~\eqref{part_c} of \cref{circle_action_characterization}.

\cref{circle_action_characterization} can be used for understanding
where \textit{not} to look for examples of free symplectic circle actions 
with contractible orbits:
{
\begin{corollary}\label{cor:coh_constraints}
Let $(N^{2n}, \Omega)$ be a closed connected symplectic manifold. 
Suppose that for every $c \in H^2(N;\Z)$,
if $\langle c, \pi_2(N)\rangle = \mathbb{Z}$, 
then there exists $j > 0$ such that ${[\Omega]}^{n-j} \wedge c_\R^j \neq 0$.
Then $(N, \Omega)$ is not symplectomorphic to a reduced space 
of a free symplectic circle action with contractible orbits.
\end{corollary}

\begin{proof}
Let $\Omega_t$ and $\Psi$ be as in 
\cref{circle_action_characterization}.
Then the function $t \mapsto [\Omega_t]^n$ from $\R$ to $H^{2n}_{\dR}(N)$ 
is $1$-periodic.  Because
\begin{equation}\label{poly_volume}
    {[\Omega_t]}^n = \sum_{j=0}^n\binom{n}{j}t^j{[\Omega]}^{n-j} \wedge c^j
\end{equation}
is polynomial in $t$, this implies that the coefficient
${[\Omega]}^{n-j} \wedge c^j$ is zero for every $j > 0$. 
\end{proof}
}
\begin{example}\label{e:contractible_examples}
By \cref{cor:coh_constraints}, the following closed connected symplectic
manifolds cannot appear as reduced spaces of free symplectic circle
actions with contractible orbits:
\begin{enumerate}
\item\label{item_aspherical} Aspherical symplectic manifolds.
\item\label{item_betti}
Symplectic manifolds with second Betti number equal to one. 
Examples include symplectic surfaces, complex projective spaces, 
complex Grassmannians, and Grassmannian manifolds associated with arbitrary 
compact connected semisimple Lie groups, as in Section 7 of \cite{caviedes}
(Section 5 of the arXiv version \cite{caviedes_arxiv}).
\end{enumerate}
\end{example}

\begin{remark}
If the reduced space $N$ is aspherical, it follows from long exact 
sequences of homotopy groups that $M$ is also aspherical. Thus, 
Example~\ref{e:contractible_examples}.\eqref{item_aspherical}
also follows from Lupton and Oprea's work~\cite{lupton_oprea},
by which a symplectic circle action on an aspherical
closed symplectic manifold must have non-contractible orbits 
(see also~\cite{atallah_remarks_on_circle_actions, ono}).
\end{remark}


\section{Open questions}\label{open_questions_section}
In this section we raise several open questions on the topology of non-Hamiltonian symplectic torus action.

\subsection{Actions with some contractible orbits}\label{open_questions_some_contractible_subsection}\ 

The constructions in McDuff~\cite{mcduff_circle_actions}, Kotschick~\cite{kotschick_contractible_orbits}, Tolman~\cite{tolman_isolated_points}, and Jang--Tolman~\cite{jang_tolman} give examples of symplectic non-Hamiltonian circle actions whose orbits are all contractible. In the smooth (non-symplectic) setting, the existence of one contractible orbit does not imply that all orbits are contractible (\cref{example_fixed_non_contractible}). Hence, it is natural to ask the following two questions (see also \cref{table_contr_questions} for a summary of the different cases):
\begin{question}\label{fixed_non_contractible_question}
Does there exist a symplectic (non-Hamiltonian) circle action on a closed
connected symplectic manifold, with fixed points, 
such that not all of its orbits are contractible?
\end{question}

\begin{question}\label{some_non_contractible_question}
Does there exist a symplectic (non-Hamiltonian) locally-free circle action on a closed
connected symplectic manifold, such that some but not all of its orbits are contractible?
\end{question}

In dimension four, the answer to both questions
is negative by \cref{corollary_tfae} (or by \cite[Prop.2]{mcduff_circle_actions} for \cref{fixed_non_contractible_question}). For free actions, the answer to \cref{some_non_contractible_question} is negative by \cref{lem:contractible}. It seems plausible that constructions
inspired by~\cite{kotschick_contractible_orbits, tolman_isolated_points,
jang_tolman} might resolve these questions positively in dimension six.

\begin{table}[ht]
\centering
\renewcommand{\arraystretch}{1.5}

\begin{tabular}{| >{\raggedright\arraybackslash}p{0.35\textwidth} | >{\centering\arraybackslash}p{0.3\textwidth} | >{\centering\arraybackslash}p{0.3\textwidth} |}
\hline
& \textbf{Has fixed points} & \textbf{No fixed points} \\
\hline \hline

\textbf{All orbits are contractible} &
McDuff~\cite{mcduff_circle_actions}, Tolman~\cite{tolman_isolated_points}, Jang--Tolman~\cite{jang_tolman} &
Kotschick~\cite{kotschick_contractible_orbits}, Tolman~\cite{tolman_isolated_points} \\
\hline

\textbf{Some but not all orbits are contractible} &
No known examples (\cref{fixed_non_contractible_question}) &
No known examples (\cref{some_non_contractible_question}) \\
\hline

\end{tabular}

\vspace{10pt}
\caption{The state of knowledge concerning the existence of symplectic non-Hamiltonian circle actions with different topological properties (in dimension 6).}
\label{table_contr_questions}
\end{table}

\subsection{Reduced spaces of actions with contractible orbits}\label{open_questions_reduced_spaces_subsection}\ 

In all known constructions of free symplectic circle actions with contractible orbits, the reduced spaces are K3 surfaces. It would be interesting to find other examples:
\begin{problem}\label{reduced_spaces_open_question}
    Find a free symplectic circle action on a 6-dimensional closed connected symplectic manifold, whose reduced space is not homeomorphic to a K3 surface.
\end{problem}
See \cref{cor:coh_constraints} for cohomological constraints for a possible example.
Inspired by Tolman and Jang's recent work~\cite{jang_tolman},
seeking non-Hamiltonian circle actions with as few as possible fixed points,
we can also ask a quantitative version of this problem:
\begin{question}\label{reduced_spaces_quantitative_open_question}
What is the smallest possible second Betti number $b_2$ of a reduced space
for a free symplectic circle action with circle-valued Hamiltonian,
with contractible orbits, on a symplectic $6$-manifold?
\end{question}
By Kotschick's example, together with 
Example~\ref{e:contractible_examples}.\eqref{item_betti}, 
the answer is somewhere in the range $2 \le b_2 \le 22$.

\subsection{Actions with non-isotropic orbits}\label{open_questions_non_isotropic_subsection}\ 

Throughout the paper, our efforts revolved around understanding
the topology of symplectic $T^{n-1}$ actions
with isotropic actions. We believe that many of our results can
be extended to actions whose orbits are not isotropic. This direction
will help getting a more complete understanding of symplectic
torus actions, and in particular may be useful for future classifications
of symplectic torus actions with non-isotropic orbits, extending the works
of Benoist~\cite{benoist} and Duistermaat-Pelayo~\cite{duistermaat_pelayo}.

In particular, we believe that \cref{complexity_one_torsion_theorem} has a natural generalization to actions with non-isotropic orbits, if one chooses the dimension of the acting torus correctly:
\begin{conjecture}\label{non_isotropic_conjecture}
Let an $m$-dimensional torus $T$ act symplectically on a $2n$-dimensional closed connected symplectic manifold $(M, \omega)$, let $2r$ be the rank of the two-form $\omega_T$ from \cref{cor:sigma_constant},
and assume that
    \begin{equation}\label{eq_conjecture}
        m \ge n + r - 1.
    \end{equation}
    Then for every non-Hamiltonian subcircle action of $T$, the orbits of the circle action and all of their iterates are non-contractible.
\end{conjecture}
\cref{complexity_one_torsion_theorem} is a special case of \cref{non_isotropic_conjecture}, when $r = 0$.

Consider the product of a complexity one space of dimension $2n-2r$ with a $2r$-dimensional symplectic torus acting on itself; it satisfies $m = n + r - 1$, and indeed every non-Hamiltonian subcircle action of $T$ has non-contractible orbits which are not torsion.

In contrast, suppose that $(M, \omega)$ is a six-dimensional symplectic manifold with a non-Hamiltonian symplectic circle action with contractible orbits (as in~\cite{jang_tolman, kotschick_contractible_orbits, mcduff_circle_actions,tolman_isolated_points}). Moreover, let $(T^2, dp \wedge dq)$ be a symplectic 2-torus, given with its natural action on itself. Then the product space $(M \times T^2, \omega + dp \wedge dq)$ is an eight-dimensional symplectic manifold, with a $T^3$ action which has a non-Hamiltonian symplectic subcircle action with contractible orbits. This example satisfies $m = n + r - 2$, so the conclusion of \cref{non_isotropic_conjecture} does not apply in this case. Since $m = n-1$ holds, matching the original statement of \cref{complexity_one_torsion_theorem}, this demonstrates why the refined dimension constraint on the acting torus (Equation~\eqref{eq_conjecture}) is necessary for actions with non-isotropic orbits.

\begin{remark}
Roughly speaking, a possible (informal!) strategy for proving \cref{non_isotropic_conjecture}, which also gives intuition for Equation~\eqref{eq_conjecture}, is as follows. We consider only the case where $m = n + r - 1$; the other cases follows easily from this case. Let a torus $T^{n + r - 1}$ act on a $2n$-dimensional symplectic manifold, 
such that $\omega_T$ (from \cref{cor:sigma_constant}) has rank $2r$.
Take a quotient by the ``symplectic part'' of the $T$-action, which is of dimension $2r$, to get a symplectic orbifold of dimension $2n - 2r$, with a $T^{n - r - 1}$ action with isotropic orbits. Since the dimension of the acting torus is one less than half the dimension of the orbifold, follow the proof of \cref{complexity_one_torsion_theorem} in the setting of orbifolds to deduce the result.
To follow through with this strategy, one would need to establish the version of Moser's trick for complexity one spaces from~\cite{centered_hamiltonians, tall_existence} (see \cref{complexity_one_isotopy}) in the setting of orbifolds. This relates to the recent classification of Hamiltonian $S^1$-actions on compact 4-orbifolds, with isolated cyclic singular points, given in~\cite{orbifold_ham_s1} by Godinho, Mwakyoma-Oliveira, and Sepe. In their paper, they suggest the idea of generalizing their classification to higher dimensional complexity one orbifolds. For this, it seems likely that a similar generalization of \cref{complexity_one_isotopy} will be needed.
\end{remark}


\appendix

\section{Equivariant formality} \label{equivariant_formality_appendix}

A Hamiltonian action of a torus $T$ on a closed symplectic manifold $(M, \omega)$ must be equivariantly formal.  
This is proved by applying Morse-Bott theory to components of the momentum map; see Atiyah-Bott~\cite{atiyah_bott} and Kirwan~\cite{kirwan}.
The converse is also true:

\begin{lemma}\label{lem:ham_if_eqformal}
Let a torus act faithfully and symplectically
on a closed connected symplectic manifold $(M, \omega)$.
If the action is equivariantly formal, then it is Hamiltonian.
\end{lemma}
This lemma follows from a work of Allday--Hauschild--Puppe~\cite[Remark (3.2) Item (3)]{allday_hauschild_puppe}. In this appendix, we give a direct proof.

We first recall the definition of equivariant formality.
Let $ET \rightarrow BT$ be the universal $T$-bundle.
For every $T$-action on a manifold $M$, its Borel construction is 
\begin{equation*}
    M_T := ET \times_T M,
\end{equation*}
and the \textbf{equivariant cohomology} of $M$ is $H_T^*(M) := H^*(M_T)$,
which is a $H^*(BT)$-module. We work with real coefficients.
The $T$-manifold $M$ is \textbf{equivariantly formal} 
if $H_T^*(M)$ is a free $H^*(BT)$-module. The Borel construction $M_T$ fibers
over $BT$ with fiber $M$. Because $T$ is connected,
the inclusion of the fiber induces a well-defined map
\begin{equation*}
    s:H^*_T(M) \to H^*(M)
\end{equation*}
which is surjective if and only if the $T$-manifold $M$ is equivariantly formal.
For more information, see Borel~\cite{borel_seminar}, Atiyah--Bott~\cite{atiyah_bott}, Berline--Vergne~\cite{berline_vergne}, Guillemin--Sternberg~\cite{gs_supersymmetry}, Kirwan~\cite{kirwan}, and Ginzburg--Guillemin--Karshon~\cite[Appendix C]{ggk_book}.

Cartan's model for equivariant cohomology 
allows us to define equivariant cohomology using differential forms. 
The coefficient ring $H^*(BT)$ is the ring of symmetric polynomials 
on the Lie algebra $\t$ of $T$;
equivalently, the symmetric tensors on $\t^*$.
A splitting $T \cong (S^1)^m$ identifies it with $\R[x_1,\dots,x_{m}]$.
The equivariant differential forms of degree $k$ are given by
\begin{equation*}
\Omega_T^k(M) := \bigoplus_{j + 2l = k} (\Omega^j(M) \otimes S^l[\t^*])^T,
\end{equation*}
where $S^l[\t^*]$ is the space of symmetric polynomials of degree $l$ on $\t$.
The equivariant differential 
\begin{equation*}
(d_T \alpha)(\xi) := d(\alpha(\xi)) + \iota_{\xi_M}(\alpha(\xi))
\end{equation*}
makes $\Omega_T^*(M)$ into a cochain complex whose cohomology is $H_T^*(M)$.
In degree $k = 2$, by definition, we have
\begin{equation}\label{equation_omega2}
\Omega_T^2(M) = {(C^{\infty}(M) \otimes \t^*)}^T \oplus {(\Omega^2(M))}^T,
\end{equation}
and for every $\mu \in {(C^{\infty}(M) \otimes \t^*)}^T$ and $\beta \in {(\Omega^2(M))}^T$, the map $s \colon H_T^2(M) \rightarrow H^2(M)$ takes $[\mu+\beta]$ to $[\beta]$.
\begin{proof}[Proof of \cref{lem:ham_if_eqformal}]
A two-cochain $\mu \oplus \beta$ in $\Omega_T^2(M)$ is closed if and only if
\begin{equation*}
(d_T(\mu \oplus \beta))(\xi) = d(\beta(\xi)) + \iota_{\xi_M}(\beta(\xi)) + d(\mu(\xi)) + \iota_{\xi_M} (\mu(\xi)) = 0
\end{equation*}
for every $\xi \in \t$. The last term vanishes,
hence for any closed two-cochain $\mu \oplus \beta$ we have
\begin{equation}\label{eq:closed_eq}
d\beta = 0, \quad \iota_{\xi_M}\beta = - d(\mu(\xi))
\end{equation}
for every $\xi \in \t$.
By assumption, the action is equivariantly formal, and therefore the map
$$ s \colon H_T^2(M) \rightarrow H^2(M) $$
is surjective. In particular, $[\omega]$ has a preimage in $H_T^2(M)$,
represented by a closed cochain $\mu \oplus \beta$ of $\Omega_T^2(M)$.
In particular, $[\omega] = [\beta]$, thus there exists an invariant one-form 
$\eta$ such that
\begin{equation}\label{eq:omega_by_eta}
\omega = \beta + d\eta.
\end{equation}
We define a map $\Phi \colon M \rightarrow \t^*$ by
\begin{equation*}
\Phi(\xi) := \mu(\xi) + \iota_{\xi_M} \eta,
\end{equation*}
where $\xi \in \t$.
Then by~\eqref{eq:closed_eq} and Cartan's formula we have
\begin{equation*}
d(\Phi(\xi)) =  d(\mu(\xi)) + d(\iota_{\xi_M} \eta)  
    = - \iota_{\xi_M}\beta + \mathcal{L}_{\xi_M} \eta - \iota_{\xi_M} d\eta.
\end{equation*}
By the invariance of $\eta$, and by~\eqref{eq:omega_by_eta}, we then have
\begin{equation*}
d(\Phi(\xi)) = - \iota_{\xi_M}\beta - \iota_{\xi_M} d\eta 
 = - \iota_{\xi_M} \omega,
\end{equation*}
so the map $\Phi$ is a momentum map for $\omega$, 
and so the action is Hamiltonian.
\end{proof}

Combining \cref{lem:ham_if_eqformal} with a recent result of Bai--Pomerleano~\cite{bai_pomerleano} (which is proved using ``hard'' holomorphic curves techniques), we obtain:

\begin{corollary}\label{cor:ham_iff_eqformal}
Let a torus act faithfully and symplectically
on a closed connected symplectic manifold $(M, \omega)$.
Then the following are equivalent: 
\begin{enumerate}
\item \label{eq1} The action is equivariantly formal over $\mathbb{Z}$.
\item \label{eq2} The action is equivariantly formal over $\mathbb{R}$.
\item \label{eq3} The action is Hamiltonian.
\end{enumerate}
\end{corollary}

\begin{proof}
\eqref{eq1} implies~\eqref{eq2} by tensoring with $\mathbb{R}$.
\eqref{eq2} implies~\eqref{eq3} by \cref{lem:ham_if_eqformal}.
\eqref{eq3} implies~\eqref{eq1} by Bai--Pomerleano's \cite[Theorem 1.3]{bai_pomerleano}.
\end{proof}


\section{Higher-dimensional torus actions} \ 
\label{weak_main_theorem_appendix}

In this appendix, we prove the following weaker version of 
\cref{thm:main}, using the description of the fundamental group 
of symplectic torus actions with coisotropic orbits given by Duistermaat
and Pelayo~\cite[Corollary 8.3]{duistermaat_pelayo}.

\begin{theorem}\label{weak_main_theorem}
    Let $T$ act symplectically on a closed connected symplectic manifold $(M, \omega)$. Assume that $\dim T \ge \frac{1}{2}\dim M$. Then the action is Hamiltonian if and only if its orbits are contractible.
\end{theorem}
\begin{proof}
    By \cref{hamiltonian_has_contractible}, if the action is Hamiltonian then its orbits are contractible. For the other direction, assume that the action is non-Hamiltonian. By \cref{non_isotropic_non_contractible_lemma}, if the action has a non-isotropic orbit, then the orbits are non-contractible. Hence, we assume that the orbits of the action are isotropic, and since $\dim T \ge \frac{1}{2}\dim M$, the orbits must be Lagrangian. We can therefore apply results from Duistermaat and Pelayo's equivariant classification of symplectic torus actions with coisotropic orbits (see~\cite{duistermaat_pelayo}).
    
    Let $P$ be the group of periods of the action, let $T_{\ham}$ be the maximal subtorus of $T$ which acts in a Hamiltonian fashion, and let $T_f$ be a complementary subtorus to $T_{\ham}$ in $T$. 
    Moreover, let $\t_f$ be the Lie algebra of $T_f$, and let ${(\t_f)}_{\mathbb{Z}}$ be its integral lattice. Then by~\cite[Proposition 8.2, Corollary 8.3]{duistermaat_pelayo}, there exists a function $g \colon P \times P \rightarrow {(\t_f)}_{\mathbb{Z}}$, such that the fundamental group of $M$ is isomorphic to ${(\t_f)}_{\mathbb{Z}} \times P$, with the multiplication rule given by
    \begin{equation*}
        (x', p') \cdot (x, p) = (x + x' - g(p, p'), p + p'),
    \end{equation*}
    for $(x, p)$ and $(x', p')$ in ${(\t_f)}_{\mathbb{Z}} \times P$. Moreover, the classes represented by the orbits of the subcircles of $T_f$ can be identified with the subgroup ${(\t_f)}_{\mathbb{Z}} \times \left\{0\right\}$. It follows that every subcircle action of $T_f$ has non-contractible orbits. Since we assumed that the action is non-Hamiltonian, then $T_f \subset T$ is non-trivial, so $T$ has a non-trivial subcircle action with non-contractible orbits, and therefore the $T$-orbits are non-contractible.
\end{proof}


\section{Seifert fibrations, and the proof in dimension four}
\label{seifert_lemma_appendix}

In this appendix, we give a direct proof of
\cref{thm:main} and \cref{complexity_one_torsion_theorem} 
in the special case when $M$ is four-dimensional. In this case, the level sets 
of the circle-valued Hamiltonian are Seifert fibrations, and the proof 
is significantly simpler.
This proof of \cref{thm:main} in dimension four is sketched in
Lalonde--McDuff--Polterovich~\cite{lalonde_mcduff_polterovich};
we give more details.

First, we recall some facts about Seifert manifolds. A \textbf{Seifert
manifold} is a smooth 3-dimensional $E$ which is given with a locally-free
circle action. Its orbit space is a 2-dimensional (effective quotient)
orbifold, and we denote it by $B$. From now on we will assume that both
$E$ and $B$ are orientable. The quotient map $E \rightarrow B$ is often
called a \textbf{Seifert fibration}. When the circle action is free,
then the orbit space $B$ is a smooth manifold, and the quotient map $E
\rightarrow B$ defines a principal $S^1$-bundle. In general, the quotient
map defines a principal $S^1$-orbibundle.

The orbit space $B$ is homeomorphic to an orientable surface, and has finitely 
many non-smooth points, which correspond to the non-free orbits of the circle 
action on $E$. The stabilizer of each non-free orbit is a finite non-trivial 
subgroup of the circle, and it acts faithfully on the normal to the orbit 
which is isomorphic to $\mathbb{R}^2$. Therefore, by the slice theorem, 
the isomorphism type of a neighbourhood of each non-free orbit is given by the 
cyclic stabilizer $\mathbb{Z}_a$, together with its faithful representation 
on $\mathbb{R}^2$, which is given by an integer $0 < b < a$ coprime to $a$. 
Here, the action of $\mathbb{Z}_a$ on $\mathbb{R}^2 \cong \mathbb{C}$ 
can be described by the action of the generator $e^{2\pi i / a}$, which acts 
by $e^{2\pi i / a} \cdot z = e^{2\pi b i / a} z$. Hence, we describe 
this isomorphism class by a pair of coprime natural integers $1 < a$ 
and $0 < b < a$. Seifert fibrations are then classified by their Seifert 
invariants given by
\begin{equation*}
    (b, g;(a_1, b_1),\dots,(a_k,b_k)),
\end{equation*}
where $(a_k, b_k)$ are the coprime pairs of the non-free orbits, $g$ is the 
genus of $B$, and $b$ is an additional integer that is needed for describing 
the fibration type. As in the smooth case, we can define an Euler class for 
the Seifert fibration, which we call the \textbf{Seifert Euler number}.
It is given by the formula
\begin{equation}\label{equation_euler_number}
    e(E \rightarrow B) = -b - \sum\frac{b_i}{a_i}.
\end{equation}
When the action is free, the sum vanishes, and $e = -b$ is the usual Euler number. Moreover, the Seifert Euler number agrees with the rational Chern number of the orbibundle, see for example~\cite[Section 1.2.1 and Remark 1.67]{orbifold_ham_s1}.

We will need the following lemma, which is a special case of \cref{non_free_torus_bundle_lemma} when the base space is of dimension two.
\begin{lemma}\label{seifert_lemma}
    Let $E \rightarrow B$ be Seifert fibration, with $E$ and $B$ closed connected and orientable, and with Seifert Euler number zero. Then the orbits of the circle action, and all of their iterates, are non-contractible in $E$.
\end{lemma}
\begin{proof}
    Let $(b, g;(a_1, b_1),\dots,(a_k,b_k))$ be the Seifert invariants of $E \rightarrow B$. By assumption, the Seifert Euler number vanishes, and therefore from Equation~\eqref{equation_euler_number} we have
    \begin{equation}\label{b_equation_euler}
        b = -\sum_{1 \le i \le k}\frac{b_i}{a_i}.
    \end{equation}
    Let $x$ be some point in $E$. The fundamental group of $E$ at $x$ is given by
    \begin{equation}\label{fund_group}
        \pi_1(E, x) = \left\langle  A_1,B_1,\dots,A_g,B_g,q_1,\dots,q_k,h\ \middle|\ 
            \begin{array}{l}
            e = [A_i, h] = [B_i, h] = [q_i, h] \\
            e = q_j^{a_j}h^{b_j} \\
            h^b = q_1  q_2  \dots  q_r  [A_1,B_1] \dots [A_g, B_g]
            \end{array}
        \right\rangle,
    \end{equation}
    where $h$ is the class of the free orbits. See~\cite[Section 5.3]{orlik_book} for more information.

    Let $m \ge 1$ be some arbitrary positive integer. We wish to show that $h^m$ is non-trivial in $\pi_1(E, x)$. Let $p>0$ be a prime number, larger than every $a_i$, and larger than $m$. Let $\mathbb{Z}_p$ be the cyclic group of order $p$ and let $\gamma$ be a generator. We define a homomorphism $f \colon \pi_1(E, x) \rightarrow \mathbb{Z}_p$ by its values on the generators
    \begin{equation*}
       f(A_i) = e, \quad f(B_i) = e, \quad f(h) = \gamma, \quad f(q_j) = \gamma^{-\frac{b_j}{a_j}}.
    \end{equation*}
    Note that the root of order $a_j$ of the element $\gamma^{-b_j}$ is well-defined, since $a_j$ and $b_j$ and coprime to $p$. Moreover, note that $f$ is indeed a well-defined homomorphism because all of the relations defined in Equation~\eqref{fund_group} are satisfied, thanks to Equation~\eqref{b_equation_euler}. Since $h^m$ is mapped by $f$ to the non-trivial element $\gamma^m$, it must be non-trivial in $\pi_1(E, x)$, and this concludes the proof.
\end{proof}

Now, using \cref{seifert_lemma}, and lemmas from~\cite{classification_non_ham_s1}, we give a short proof of \cref{thm:main} and \cref{complexity_one_torsion_theorem} in dimension four. Note that in dimension four, we are interested in circle actions whose orbits are always isotropic.

\begin{proof}[Proof of \cref{thm:main} and \cref{complexity_one_torsion_theorem} in dimension four]
    If the action is Hamiltonian, then its orbits are contractible by \cref{hamiltonian_has_contractible}. Thus, we are left with proving the other direction, that the orbits of a non-Hamiltonian action, and all of their iterates, are non-contractible. By \cref{perturbation_lemma}, we can assume without loss of generality that $\omega$ is integral, and thus we have a circle-valued momentum map $\mu \colon M \rightarrow \mathbb{R}/P$, where $P$ is the group of periods of $\iota_X \omega$ given by
    \begin{equation*}
        P := \left\{\ \int_A \iota_X \omega \mid A \in H_1(M, \mathbb{Z})\right\},
    \end{equation*}
    where $X$ is the generating vector field of the circle action. Let $y$ be some value in $\mathbb{R}/P$. By~\cite[Lemma 2.2]{classification_non_ham_s1}, the level set $\mu^{-1}(y)$ is connected (see also \cref{connected_level_sets_lemma}). By~\cite[Proposition 1]{mcduff_circle_actions} the circle action has no fixed points, and thus the circle action on $\mu^{-1}(y)$ is locally-free (see also \cref{infinitesimal_splitting_lemma}). 
    Therefore, the level set $\mu^{-1}(y)$ is a closed connected Seifert manifold. Following~\cite[Lemma 2.4]{classification_non_ham_s1}, the Seifert Euler number of the Seifert fibration $\mu^{-1}(y) \rightarrow \mu^{-1}(y)/S^1$ must vanish since it is equal to the derivative of the Duistermaat-Heckman function $f_{\DuHe}\colon \mathbb{R}/P \rightarrow \mathbb{R}$, and this derivative is constant since there are no fixed points, and thus vanishes since the domain of $f_{\DuHe}$ is a circle $\mathbb{R}/P$ (see also \cref{chern_class_vanishes_proposition}). Hence, by \cref{seifert_lemma}, the orbits and their iterates are non-contractible in the Seifert manifold $\mu^{-1}(y)$ (see also \cref{non_free_torus_bundle_lemma}).

    Because there are no fixed points, the circle-valued momentum map $\mu$ is a submersion, and therefore, by Ehresmann's theorem, it defines a locally trivial fibration. Thus, we have the following long exact sequence of homotopy groups
    \begin{equation*}
        \cdots \rightarrow \pi_2(\mathbb{R}/P, y) \rightarrow \pi_1(\mu^{-1}(y), x) \rightarrow \pi_1(M, x) \rightarrow \cdots,
    \end{equation*}
    where $x$ is a point in the Seifert manifold $\mu^{-1}(y)$.
    Since $\mathbb{R}/P$ is homeomorphic to a circle, we have that $\pi_2(\mathbb{R}/P, y)$ vanishes, and so the map $\pi_1(\mu^{-1}(y), x) \rightarrow \pi_1(M, x)$ is injective. Therefore, a loop in $\mu^{-1}(y)$ is contractible in $M$ if and only if it is contractible in $\mu^{-1}(y)$, and this concludes the proof.
\end{proof}


\section{Perturbations of invariant symplectic forms}
\label{perturbation_appendix}

In this appendix, we describe how to perturb an invariant symplectic
form to be rational, while preserving some of the properties of the
action with respect to the symplectic form.

\begin{lemma}\label{perturbation_lemma}
Let $T$ act symplectically on a closed connected symplectic manifold 
$(M, \omega)$, with isotropic orbits. Let $T_H$ be a subtorus that acts
in a Hamiltonian fashion, 
and let $C_1,\dots,C_k \subset T$ 
be subcircles that act in a non-Hamiltonian fashion. Then there exists 
an integral $T$-invariant symplectic form $\omega'$ on $M$, such that
\begin{enumerate}
\item the orbits are isotropic with respect to $\omega'$,
\item $T_H$ acts in a Hamiltonian fashion with respect to $\omega'$, and
\item each $C_j$ acts in a non-Hamiltonian fashion with respect to $\omega'$.
\end{enumerate} 
\end{lemma}

\begin{proof}
Denote by $H_2(M)$ the singular homology (over $\Z$),
by $H^2_{\dR}(M)$ the de Rham cohomology,
and by $H^2_{\dR}(M)_\Q$ the image of the natural inclusion map
$$H^2(M;\Q) \to H^2_{\dR}(M).$$
Then $H^2_{\dR}(M)_{\Q}$ is a vector space over $\Q$,
it consists of those classes $[\beta]$
whose evaluation on every element of $H_2(M)$ is in $\Q$,
and $\on{span}_{\R} H^2_{\dR}(M)_{\Q} = H^2_{\dR}(M)$.

Let $$i^* \colon H^2_{\dR}(M) \to H^2_{\dR}(T)$$ 
be the map on de Rham cohomology that is induced by the orbit map 
$\rho_x \colon a \mapsto a \cdot x$ for some $x \in M$.  
This map is independent of the choice of~$x$, 
since every two orbit maps are homotopic.  
For any invariant closed two-form $\beta$ on $M$,
the $T$-orbits in $M$ are isotropic with respect to $\beta$ 
if and only if $i^*[\beta]=0$.

For each subcircle $C$ of $T$ and closed curve $\gamma \colon [0,1] \to M$
in $M$, define $f_{C,\gamma} \colon C \times S^1 \to M$ by 
$(a,e^{2\pi i t}) \mapsto a \cdot \gamma(t)$, 
and let 
$$f_{C,\gamma}^* \colon H^2_{\dR}(M) \to H^2_{\dR}(C \times S^1)$$ 
be the induced map on de Rham cohomology.
This map is independent of the choice of $\gamma$ in its homology class
$[\gamma] \in H_1(M)$.
For any closed invariant two-form $\beta$ on $M$,
the action of $C$ on $M$ admits a momentum map with respect to $\beta$ 
if and only if $f_{C,\gamma}^*[\beta]=0$ for all $[\gamma] \in H_1(M)$.

Thus, 
$$[\omega] \in W_\R := 
 \big(\ker i^*\big) \cap 
\Big( \bigcap_{\substack{C \subset T_H \\ [\gamma] \in H_1(M)}} 
 \ker f_{C,\gamma}^* \Big),$$
and for each $C_j$, we can choose $[\gamma_j] \in H_1(M)$
such that $f_{C_j,\gamma_j}^*[\omega] \neq 0$.

Because each of the maps $i^*$ and $f_{C,\gamma}^*$ 
can also be defined on (singular) cohomology over $\Q$,
the $\Q$ vector space 
$$W_\Q := W_\R \cap H^2_{\dR}(M)_{\Q}$$
satisfies $\on{span}_{\R} W_\Q = W_\R$.
Let $B_1,\ldots,B_m$ be a basis to $W_\Q$ (over $\Q$).
For each $j$, let $\beta_j$ be a closed differential 2-form on $M$
such that $B_j = [\beta_j]$.
After averaging each $\beta_j$ by the $T$-action,
we can assume that the forms $\beta_j$ are $T$-invariant.
The map $(t_1,\ldots,t_m) \mapsto \sum t_j B_j$
is an isomorphism of real vector spaces from $\R^m$ to $W_\R$, 
and the map $(t_1,\ldots,t_m) \mapsto [\omega] + \sum t_j B_j$
is an isomorphism of \emph{affine} vector spaces from $\R^m$ to $W_\R$.

Because $M$ is compact, the set of $\tau = (t_1,\ldots,t_m)$ such that 
$\omega_\tau := \omega + \sum t_j \beta_j$ is non-degenerate 
and such that $f_{C_j,\gamma_j}^* [\omega_\tau] \neq 0$
for all $j \in \{1,\ldots,k\}$ is open in $W_\R$. 
Because this open set contains $\tau=0$, it is nonempty,
so it meets the dense subset $W_\Q$ of $W_\R$,
say, in $\omega_\tau$.  
Then $\omega_\tau$ is non-degenerate, 
the $C_j$ action on $M$ is non-Hamiltonian with respect to $\omega_\tau$
for all $j \in \{1,\ldots,k\}$, and $[\omega_\tau] \in W_\Q$.
And since $\omega_\tau$ is closed and $T$-invariant
and $[\omega_\tau] \in W_\R$,
the $T$-orbits are isotropic with respect to $\omega_\tau$
and the $T_H$ action is Hamiltonian with respect to $\omega_\tau$.
By rescaling $\omega_{\tau}$, we get an integral two-form as required. 
\end{proof}

\printbibliography

\end{document}